\documentclass{amsart}

\usepackage{amsmath} 
\usepackage{amssymb}

\newtheorem{theorem}{Theorem}[section] 
\newtheorem{claim}{Claim}[theorem]
\newtheorem{lemma}[theorem]{Lemma} 
\newtheorem{proposition}[theorem]{Proposition} 
\newtheorem{observation}[theorem]{Observation} 
\newtheorem{corollary}[theorem]{Corollary} 

\theoremstyle{definition}
\newtheorem{definition}[theorem]{Definition}

\newtheorem{problem}[theorem]{Problem}

\theoremstyle{remark}
\newtheorem{remark}[theorem]{Remark}

\newtheorem{conclusion}[theorem]{Conclusion}

\numberwithin{equation}{section}

\newcommand{\conc}{{}^\frown\!}
\newcommand{\lh}{{\rm lh}\/}
\newcommand{\rest}{{\restriction}}
\newcommand{\vtl}{\vartriangleleft}
\newcommand{\dom}{{\rm dom}} 
\newcommand{\rk}{{\rm rk}} 

\newcommand{\comp}{\circ} 
\newcommand{\suc}{{\rm succ}}
\newcommand{\mrot}{{\rm root}}

\newcommand{\bbC}{{\mathbb C}}

\newcommand{\bbD}{{\mathbb D}}

\newcommand{\cH}{{\mathcal H}}

\newcommand{\cO}{{\mathcal O}}
\newcommand{\bbP}{{\mathbb P}}

\newcommand{\bbQ}{{\mathbb Q}}
\newcommand{\dbQ}{{\name{\mathbb Q}}}

\newcommand{\cT}{{\mathcal T}}
\newcommand{\bT}{{\mathbb T}}

\newcommand{\cf}{{\rm cf}\/}

\newcommand{\st}{{\bf st}} 

\newcommand{\vare}{\varepsilon}

\newcommand{\forces}{\Vdash} 
\newcommand{\bV}{{\bf V}} 
\newcommand{\lesdot}{\mathrel{\mathord{<}\!\!\raise 
0.8 pt\hbox{$\scriptstyle\circ$}}} 

\newcommand{\pr}{{\rm pr}}

\newcommand{\auxgame}{{\Game^{{\rm aux}}_{\bar{Y}}(p,\bar{q},\bbQ,\leq,
    \bar{\leq}_\pr,D)}}  
\newcommand{\auxzero}{{\Game^{{\rm aux}}_{\bar{Y}}}} 
 
\newcommand{\auxk}{{\Game^{{\rm aux}}_{\bar{Y}^\xi}}} 
\newcommand{\masgame}{{\Game^{{\rm main}}_{\bar{Y}}(p,\bbQ,\leq,
    \bar{\leq}_\pr,D)}}  
\newcommand{\masgamexi}{{\Game^{{\rm main}}_{\bar{Y}^\xi}(r_\delta(\xi),
    \name{\bbQ}_\xi,\leq,\bar{\leq}_\pr,D^{\bV^{\bbP_\xi}})}} 
\newcommand{\masgamezero}{{\Game^{{\rm main}}_{\bar{Y}^\xi}}} 

\newcommand{\qone}{{{\mathbb Q}^2_\lambda}}
\newcommand{\qtwo}{{{\mathbb Q}^1_\lambda}}
\newcommand{\qthree}{{{\mathbb Q}^3_\lambda}}
\newcommand{\qfour}{{{\mathbb Q}^4_\lambda}}
\newcommand{\qell}{{{\mathbb Q}^\ell_\lambda}}

\newcount\skewfactor
\def\mathunderaccent#1#2 {\let\theaccent#1\skewfactor#2
\mathpalette\putaccentunder}
\def\putaccentunder#1#2{\oalign{$#1#2$\crcr\hidewidth
\vbox to.2ex{\hbox{$#1\skew\skewfactor\theaccent{}$}\vss}\hidewidth}}
\def\name{\mathunderaccent\tilde-3 }

\begin{document}

\title[More about $\lambda$--support iterations]{More about
  $\lambda$--support iterations of $({<}\lambda)$--complete forcing notions}  

\author{Andrzej Ros{\l}anowski}
\address{Department of Mathematics\\
 University of Nebraska at Omaha\\
 Omaha, NE 68182-0243, USA}
\email{roslanow@member.ams.org}
\urladdr{http://www.unomaha.edu/logic}

\author{Saharon Shelah}
\address{Einstein Institute of Mathematics\\
Edmond J. Safra Campus, Givat Ram\\
The Hebrew University of Jerusalem\\
Jerusalem, 91904, Israel\\
 and  Department of Mathematics\\
 Rutgers University\\
 New Brunswick, NJ 08854, USA}
\email{shelah@math.huji.ac.il}
\urladdr{http://shelah.logic.at}

\thanks{We would like to thank the referee for valuable comments and
  suggestions.\\
Both authors acknowledge support from the United States-Israel
Binational Science Foundation (Grant no. 2006108). This is publication
942 of the second author.}     

\date{September 2012}

\begin{abstract}
This article continues Ros{\l}anowski and Shelah \cite{RoSh:655,
  RoSh:860, RoSh:777, RoSh:888, RoSh:890} and we introduce here a new
property of $({<}\lambda)$--strategically complete forcing notions which
implies that their $\lambda$--support iterations do not collapse $\lambda^+$
(for a strongly inaccessible cardinal $\lambda$).  
\end{abstract}

\maketitle

\section{Introduction}
The systematic studies of iterations with uncountable supports which do not
collapse cardinals were intensified with articles Shelah \cite{Sh:587,
  Sh:667}. Those works started the development of a theory parallel to that
of ``proper forcing in CS iterations'', but the drawback there was that the
corresponding properties were more like those in the case of ``not adding
new reals in CS iterations of proper forcings''. If we want to investigate
cardinal characteristics associated with ${}^\lambda\lambda$ (in a manner it
was done for cardinal characteristics of the continuum), we naturally are
interested in iterating forcing notions which do add new elements of
${}^\lambda\lambda$. The study of $\lambda$--support iterations of such
forcing notions (for an uncountable cardinal $\lambda$) has a quite long
history already. For instance, Kanamori \cite{Ka80} considered iterations of
$\lambda$--Sacks forcing notion (similar to the forcing
$\bbQ^{2,\bar{E}}$; see Definition \ref{defFilter} and Remark
\ref{oldfor}) and he proved that under some circumstances these iterations
preserve $\lambda^+$. Fusion properties of iterations of other tree--like
forcing notions were used in Friedman and Zdomskyy \cite{FrZd10} and
Friedman, Honzik and Zdomskyy \cite{FrHoZdxx}. In particular, they showed
that $\lambda$--support iterations of a close relative of $\bbQ^2_\lambda$
from Definition \ref{def2.1} do not collapse $\lambda^+$.  Several
conditions ensuring that $\lambda^+$ is not collapsed in $\lambda$--support
iterations were introduced in a series of previous works Ros{\l}anowski and
Shelah \cite{RoSh:655, RoSh:860, RoSh:777, RoSh:888, RoSh:890}. Also
Eisworth \cite{Ei03} introduced a condition of this type. Each of those
conditions was meant to be applicable to some natural forcing notions adding
a new member of ${}^\lambda\lambda$ without adding new elements of
${}^{{<}\lambda} \lambda$. In some sense, they explained why the relevant
forcings can be iterated (without collapsing cardinals).

In the present paper we introduce {\em semi--pure properness\/} (Definition 
\ref{def1.3}) and we show that for an inaccessible cardinal $\lambda$,
$\lambda$--support iterations of semi--purely proper forcing notions are
proper in the standard sense (Theorem \ref{thm1.6}). The cases of
successor $\lambda$ and/or weakly inaccessible $\lambda$ will be
treated in a subsequent paper \cite{RoSh:1001}. 

The semi--pure properness is designed to cover the forcing notion $\qone$
mentioned above (and its relatives given in \ref{def2.1}, \ref{defFilter}),
but we hope it is much more general. This property has a flavor of {\em
  fuzzy properness over quasi--diamonds\/} of \cite[Definition
A.3.6]{RoSh:777} and even more so of {\em being reasonably merry\/} of
\cite[Definition 6.3]{RoSh:888}.  There is also some similarity with {\em
  pure $B^*$--boundedness\/} of \cite[Definition 2.2]{RoSh:888}. However,
the exact relationships between these and other properness conditions are
not clear.  

While there are some similarities between conditions studied so far, we are
far from the state that was achieved for CS iterations and the concept of
properness.  The considered properties are (unfortunatelly) tailored to fit
particular forcing notions and they do not provide any satisfactory general
framework covering all examples. The search for the ``right'' notion of
$\lambda$--propernes is still far from being completed.

Basic definitions concerning strategically complete forcing notions, their
iterations and trees of conditions are reminded in the further part of the
Introduction. In the second section of the paper we prove our Iteration
Theorem \ref{thm1.6} and in the following section we present the forcing
notions to which this theorem applies. Some special properties of and
relationships between the forcings from the third section are investigated
in the fourth section.

\subsection{Notation}
Our notation is rather standard and compatible with that of classical
textbooks (like Jech \cite{J}). However, in forcing we keep the older
convention that {\em a stronger condition is the larger one}. 

\begin{enumerate}
\item Ordinal numbers will be denoted be the lower case initial letters of
the Greek alphabet ($\alpha,\beta,\gamma,\delta\ldots$) and also by $i,j$
(with possible sub- and superscripts). 

Cardinal numbers will be called $\kappa,\lambda$; {\bf $\lambda$ will
  be always assumed to be a regular uncountable cardinal such that
  $\lambda^{<\lambda}=\lambda$\/}; in most instances $\lambda$ is even
assumed to be {\bf strongly inaccessible}. 

Also, $\chi$ will denote a {\em sufficiently large\/} regular cardinal;
$\cH(\chi)$ is the family of all sets hereditarily of size less than
$\chi$. Moreover, we fix a well ordering $<^*_\chi$ of $\cH(\chi)$.

\item We will consider several games of two players. One player will be
called {\em Generic\/} or {\em Complete\/} or just {\em COM\/}, and we
will refer to this player as ``she''. Her opponent will be called {\em
Antigeneric\/} or {\em Incomplete} or just {\em INC\/} and will be
referred to as ``he''.

\item For a forcing notion $\bbP$, all $\bbP$--names for objects in
  the extension via $\bbP$ will be denoted with a tilde below (e.g.,
  $\name{\tau}$, $\name{X}$), and $\name{G}_\bbP$ will stand for the
  canonical $\bbP$--name for the generic filter in $\bbP$. The weakest
  element of $\bbP$ will be denoted by $\emptyset_\bbP$ (and we will always
  assume that there is one, and that there is no other condition  equivalent
  to it). 

  By ``$\lambda$--support iterations'' we mean iterations in which domains
  of conditions are of size $\leq\lambda$. However, on some occasions we
  will pretend that conditions in a $\lambda$--support iteration
  $\bar{\bbQ}=\langle\bbP_\zeta,\name{\bbQ}_\zeta:\zeta< \zeta^* \rangle$
  are total functions on $\zeta^*$ and for $p\in\lim(\bar{\bbQ})$ and
  $\alpha\in\zeta^*\setminus\dom(p)$ we will let
  $p(\alpha)=\name{\emptyset}_{\name{\bbQ}_\alpha}$.

\item A filter on $\lambda$ is a non-empty family of subsets of $\lambda$
  closed under supersets and intersections and do not containing
  $\emptyset$. A filter is $({<\lambda})$--complete if it is closed under
  intersections of ${<}\lambda$ members. (Note: we do allow principal
  filters or even $\{\lambda\}$.)

For a filter $D$ on $\lambda$, the family of all $D$--positive subsets
of $\lambda$ is called $D^+$. (So $A\in D^+$ if and only if $A\subseteq
\lambda$ and $A\cap B\neq\emptyset$ for all $B\in D$.) By a normal filter on
$\lambda$ we mean {\em proper uniform\/} filter closed under diagonal
intersections.  

\item By a {\em sequence\/} we mean a function whose domain is a set of
  ordinals. For two sequences $\eta,\nu$ we write $\nu\vtl\eta$ whenever
  $\nu$ is a proper initial segment of $\eta$, and $\nu \trianglelefteq\eta$
  when either $\nu\vtl\eta$ or $\nu=\eta$.  The length of a sequence $\eta$
  is the order type of its domain and it is denoted by $\lh(\eta)$.

\item  A tree is a $\vtl$--downward closed set of sequences. A complete
  $\lambda$--tree is a tree $T\subseteq {}^{<\lambda}\lambda$ such that
  every $\vtl$-chain of size less than $\lambda$ has an $\vtl$-bound in $T$
  and for each $\eta\in T$ there is $\nu\in T$ such that $\eta\vtl\nu$.

Let $T\subseteq {}^{<\lambda}\lambda$ be a tree. For $\eta\in T$ we let 
\[\suc_T(\eta)=\{\alpha<\lambda:\eta\conc\langle\alpha\rangle\in T\}\quad
\mbox{ and }\quad (T)_\eta=\{\nu\in T:\nu\vtl\eta\mbox{ or }\eta
\trianglelefteq \nu\}.\]
We also let $\mrot(T)$ be the shortest $\eta\in T$ such that
$|\suc_T(\eta)|>1$ and $\lim_\lambda(T)=\{\eta\in{}^\lambda\lambda: (\forall
\alpha<\lambda)(\eta\rest\alpha\in T)\}$.
\end{enumerate}

\subsection{Background on trees of conditions}

\begin{definition}
\label{comp}
Let $\bbP$ be a forcing notion.
\begin{enumerate}
\item For an ordinal $\gamma$ and a condition $r\in\bbP$, let
  $\Game_0^\gamma(\bbP,r)$ be the following game of two players, {\em
    Complete} and  {\em Incomplete}:    
\begin{quotation}
\noindent the game lasts at most $\gamma$ moves and during a play the
players construct a sequence $\langle (p_i,q_i): i<\gamma\rangle$ of pairs
of conditions from $\bbP$ in such a way that 
\[(\forall j<i<\gamma)(r\leq p_j\leq q_j\leq p_i)\]
and at the stage $i<\gamma$ of the game, first Incomplete chooses $p_i$ and
then Complete chooses $q_i$.   
\end{quotation}
Complete wins if and only if for every $i<\gamma$ there are legal moves for
both players. 
\item We say that the forcing notion $\bbP$ is {\em strategically
$({<}\gamma)$--complete\/} ({\em strategically $({\leq}\gamma)$--complete},
respectively) if Complete has a winning strategy in the game 
$\Game_0^\gamma(\bbP,r)$ (in the game $\Game_0^{\gamma+1}(\bbP,r)$,
respectively) for each condition $r\in\bbP$.  
\item Let a model $N\prec (\cH(\chi),\in,<^*_\chi)$ be such that
${}^{<\lambda} N\subseteq N$, $|N|=\lambda$ and $\bbP\in N$. We say that a 
condition $p\in\bbP$ is {\em $(N,\bbP)$--generic in the standard sense\/}
(or just: {\em $(N,\bbP)$--generic\/}) if for every $\bbP$--name
$\name{\tau}\in N$ for an ordinal we have $p\forces$`` $\name{\tau}\in N$
''. 
\item $\bbP$ is {\em $\lambda$--proper in the standard sense\/} (or just:
{\em $\lambda$--proper\/}) if there is $x\in \cH(\chi)$  such that for
every model $N\prec (\cH(\chi),\in,<^*_\chi)$ satisfying  
\[{}^{<\lambda} N\subseteq N,\quad |N|=\lambda\quad\mbox{ and }\quad\bbP,x
  \in N, \]
and every condition $q\in N\cap\bbP$ there is an $(N,\bbP)$--generic
condition $p\in\bbP$ stronger than $q$.
\end{enumerate}
\end{definition}

\begin{definition}
[Compare {\cite[Def. A.1.7]{RoSh:777}}, see also {\cite[Def. 2.2]{RoSh:860}}]
\label{dA.5}
\begin{enumerate}
\item Let $\gamma$ be an ordinal, $\emptyset\neq w\subseteq \gamma$.
{\em A $(w,1)^\gamma$--tree\/} is a pair $\cT=(T,\rk)$ such that  
\begin{itemize}
\item $\rk:T\longrightarrow w\cup\{\gamma\}$, 
\item if $t\in T$ and $\rk(t)=\vare$, then $t$ is a sequence $\langle
(t)_\zeta: \zeta\in w\cap\vare\rangle$, 
\item $(T,\vtl)$ is a tree with root $\langle\rangle$ and 
\item if $t\in T$, then there is $t'\in T$ such that $t\trianglelefteq t'$
  and $\rk(t')=\gamma$.
\end{itemize}
\item If, additionally, $\cT=(T,\rk)$ is such that every chain in $T$ has a
  $\vtl$--upper bound it $T$, we will call it  {\em a standard
    $(w,1)^\gamma$--tree\/}  

We will keep the convention that $\cT^x_y$ is $(T^x_y,\rk^x_y)$.
\item Let $\bar{\bbQ}=\langle\bbP_i,\name{\bbQ}_i:i<\gamma\rangle$ be a
$\lambda$--support iteration. {\em A tree of conditions in $\bar{\bbQ}$} is
a system $\bar{p}=\langle p_t:t\in T\rangle$ such that  
\begin{itemize}
\item $(T,\rk)$ is a $(w,1)^\gamma$--tree for some $w\subseteq
\gamma$, 
\item $p_t\in\bbP_{\rk(t)}$ for $t\in T$, and
\item if $s,t\in T$, $s\vtl t$, then $p_s=p_t\rest\rk(s)$. 
\end{itemize}
If, additionally, $(T,\rk)$ is a standard tree, then $\bar{p}$ is called
{\em a standard tree of conditions}.  
\item Let $\bar{p}^0,\bar{p}^1$ be trees of conditions in $\bar{\bbQ}$,
  $\bar{p}^i=\langle p^i_t:t\in T\rangle$. We write $\bar{p}^0\leq
  \bar{p}^1$ whenever for each $t\in T$ we have $p^0_t\leq p^1_t$.   
\end{enumerate}
\end{definition}

Note that our standard trees and trees of conditions are a special case of
that \cite[Def. A.1.7]{RoSh:777} when $\alpha=1$. 

\section{Semi-purity and iterations}
In this section we introduce a new property of $({<}\lambda)$--complete
forcing notions: {\em semi--pure properness\/}. Then we prove that if
$\lambda$ is strongly inaccessible, then $\lambda$--support iterations of
semi--pure proper forcing notions are proper in the standard sense (so they
preserve stationarity of relevant sets and do not collapse $\lambda^+$). 

\begin{definition}
  \label{def1.1}
Let $f:\lambda\longrightarrow\lambda+1$. {\em A forcing notion with
  $f$--complete semi-purity\/} is a triple $(\bbQ,\leq,\bar{\leq}_\pr)$ such
that $\bar{\leq}_\pr=\langle\leq^\alpha_\pr:\alpha<\lambda\rangle$ and 
$\leq,\leq^\alpha_\pr$ are transitive and reflexive (binary) relations on
$\bbQ$ satisfying for each $\alpha<\lambda$: 
\begin{enumerate}
\item[(a)] ${\leq^\alpha_\pr}\subseteq {\leq}$,
\item[(b)] $(\bbQ,\leq)$ is strategically $({<}\lambda)$--complete and
  $(\bbQ,\leq^\alpha_\pr)$ is strategically $({\leq}\kappa)$--complete
  for all infinite cardinals $\kappa<f(\alpha)$. 
\end{enumerate}
If $(\bbQ,\leq,\bar{\leq}_\pr)$ is a forcing notion with semi-purity, then
all our forcing terms (like ``forces'', ``name'', etc) refer to
$(\bbQ,\leq)$. The relations $\leq^\alpha_\pr$ have an auxiliary character
only and if we want to refer to them we add ``$\alpha$--purely'' (so
``stronger'' refers to $\leq$ and ``$\alpha$--purely stronger'' refers to
$\leq^\alpha_\pr$).  
\end{definition}

\begin{remark}
  Note that unlike in \cite[Definition 2.1]{RoSh:888}, in semi-purity we do
  not require any kind of pure decidability.
\end{remark}

\begin{definition}
  \label{def1.3}
Let $f:\lambda\longrightarrow\lambda+1$ and let $(\bbQ,\leq,\bar{\leq}_\pr)$
be a forcing notion with $f$--complete semi-purity. Suppose that $D$ is a
normal filter on $\lambda$ (e.g., the club filter). 
\begin{enumerate}
\item A sequence $\bar{Y}=\langle Y_\alpha:\alpha<\lambda\rangle$ is called
  {\em an indexing sequence\/} whenever $\emptyset\neq Y_\alpha\subseteq
  {}^\alpha\lambda$ and $|Y_\alpha|<\lambda$ for each $\alpha<\lambda$. 
\item For an indexing sequence $\bar{Y}$, a system $\bar{q}=\langle
  q_{\alpha,\eta}:\alpha<\lambda\ \&\ \eta\in Y_\alpha\rangle\subseteq \bbQ$
  and a condition $p\in\bbQ$ we define a game $\auxgame$ between two
  players, COM and INC as follows. A play of $\auxgame$ lasts $\lambda$
  steps during which the players choose successive terms of a sequence
  $\langle (r_\alpha,A_\alpha,\eta_\alpha,r_\alpha'): \alpha<\lambda
  \rangle$. These terms are chosen so that  
\begin{enumerate}
\item[(a)] $r_\alpha,r_\alpha'\in\bbQ$, $A_\alpha\in D$, $\eta_\alpha\in
  {}^\alpha\lambda$ and for $\alpha<\beta<\lambda$:  
\[p=r_0\leq r_\alpha\leq r_\alpha'\leq r_\beta\quad \mbox{ and }\quad
A_\beta\subseteq A_\alpha\quad \mbox{ and }\quad \eta_\alpha\vtl
\eta_\beta,\] 
\item[(b)] at a stage $\alpha$ of the play, first COM chooses
  $(r_\alpha,A_\alpha,\eta_\alpha)$ and then INC picks $r_\alpha'\geq
  r_\alpha$. 
\end{enumerate}
At the end, COM wins the play $\langle (r_\alpha,A_\alpha,\eta_\alpha,
r_\alpha'): \alpha<\lambda\rangle$ if and only if both players had always
legal moves (so the play really lasted $\lambda$ steps) and 
\begin{enumerate}
\item[$(\odot)$] if $\gamma\in \mathop{\triangle}\limits_{\alpha<\lambda}
  A_\alpha$ is limit, then $\eta_\gamma\in Y_\gamma$ and
  $q_{\gamma,\eta_\gamma}\leq_\pr^\gamma r_\gamma$.  
\end{enumerate}
\item If COM has a winning strategy in $\auxgame$ then we say that {\em the
    condition $p$ is aux-generic over $\bar{q},D$}. 
\item Let $\bar{Y}$ be an indexing sequence and $p\in\bbQ$. A game
  $\masgame$ between two players, Generic and Antigeneric, is defined as
  follows. A play of the game lasts $\lambda$ steps during which the 
  players construct a sequence $\langle \bar{p}^\alpha,
  \bar{q}^\alpha:\alpha<\lambda\rangle$. At stage $\alpha<\lambda$ of the
  play, first Generic chooses a system $\bar{p}^\alpha=\langle
  p_{\alpha,\eta}:\eta\in Y_\alpha\rangle$ of pairwise incompatible
  conditions from $\bbQ$. Then Antigeneric answers by picking a system
  $\bar{q}^\alpha=\langle q_{\alpha,\eta}: \eta\in Y_\alpha\rangle$ of
  conditions from $\bbQ$ satisfying 
\[p_{\alpha,\eta}\leq^\alpha_\pr q_{\alpha,\eta}\qquad \mbox{ for all }\eta\in
Y_\alpha.\]
At the end, Generic wins the play $\langle \bar{p}^\alpha,\bar{q}^\alpha:
\alpha<\lambda\rangle$ if and only if, letting $\bar{q}=\langle
q_{\alpha,\eta}:\alpha<\lambda\ \&\ \eta\in Y_\alpha\rangle$, 
\begin{enumerate}
\item[$(\boxdot)$] there is an aux-generic condition $p^*\geq p$ over 
  $\bar{q},D$. 
\end{enumerate}
\item A forcing notion $\bbQ$ is {\em $f$--semi-purely proper over an
  indexing sequence $\bar{Y}$ and a filter $D$\/} if for some sequence 
  $\bar{\leq}_\pr$ of binary relations on $\bbQ$,
  $(\bbQ,\leq,\bar{\leq}_\pr)$ is a forcing with the $f$--complete
  semi-purity and for every $p\in \bbQ$ Generic has a winning strategy in
  $\masgame$. We then say that {\em the sequence $\bar{\leq}_\pr$ witnesses
    the semi-pure properness of $\bbQ$}.
\item If $D$ is the club filter on $\lambda$, then we omit it and we write
  $\Game^{\rm main}_{\bar{Y}}(p,\bbQ,{\leq},\bar{\leq}_\pr)$ etc. If
  $\leq^\alpha_\pr=\leq_\pr$ for all $\alpha<\lambda$, then we write
  $\leq_\pr$ instead of $\bar{\leq}_\pr$, like in $\Game^{\rm
    main}_{\bar{Y}}(p,\bbQ, {\leq},{\leq}_\pr)$. If $f(\alpha)=\lambda$ for
  all $\alpha$, then we write $\lambda$ instead of $f$ (in phrases like
  $\lambda$--complete semi--purity etc).
\end{enumerate}
\end{definition}

\begin{observation}
If $f,g:\lambda\longrightarrow \lambda+1$ and $f\leq g$, then
``$g$--semi-purely proper'' implies ``$f$--semi-purely proper''.
\end{observation}

The proof of the following proposition may be considered as an introduction 
to the more complicated and general proof of Theorem \ref{thm1.6} dealing
with the iterations. 

\begin{proposition}
   \label{prop1.4}
   Assume that $f:\lambda\longrightarrow\lambda+1$,
   $\omega+\alpha<f(\alpha)$ for $\alpha<\lambda$ and $D$ is a normal filter
   on $\lambda$. Let $\bar{Y}=\langle Y_\alpha:\alpha<\lambda\rangle$ be an
   indexing sequence. If a forcing notion $\bbQ$ is $f$--semi-purely proper
   over $\bar{Y},D$, then it is $\lambda$--proper in the standard sense.
\end{proposition}

\begin{proof}
Let $\bar{\leq}_\pr$ be a sequence witnessing the semi-pure properness of 
$\bbQ$. Assume $N\prec (\cH(\chi),\in,<^*_\chi)$ satisfies   
\[{}^{<\lambda} N\subseteq N,\quad |N|=\lambda\quad\mbox{ and }\quad
(\bbQ,\leq,\bar{\leq}_\pr),\bar{Y},D\ldots \in N. \]
Let $p\in N\cap\bbQ$. Fix a winning strategy $\st\in N$ of Generic in
$\masgame$ and pick a list $\langle\name{\tau}_\alpha:\alpha<\lambda\rangle$
of all $\bbQ$--names for ordinals from $N$. 

Consider a play of $\masgame$ in which Generic uses $\st$ and Antigeneric
chooses his answers as follows. At stage $\alpha<\lambda$ of the play, after
Generic played $\bar{p}^\alpha=\langle p_{\alpha,\eta}:\eta\in
Y_\alpha\rangle$, Antigeneric picks the $<^*_\chi$--first sequence
$\bar{q}^\alpha= \langle q_{\alpha,\eta}:\eta\in Y_\alpha\rangle$ such that
for each $\eta\in Y_\alpha$:
\begin{enumerate}
\item[$(*)_\eta$]  $p_{\alpha,\eta}\leq^\alpha_\pr q_{\alpha,\eta}$,
\item[$(**)_\eta$] if $\beta<\alpha$ and there is a condition $q$
  $\alpha$--purely stronger than $q_{\alpha,\eta}$ and forcing a value to 
  $\name{\tau}_\beta$, then $q_{\alpha,\eta}$ already forces a value to
  $\name{\tau}_\beta$.  
\end{enumerate}
Note that since $(\bbQ,\leq^\alpha_\pr)$ is strategically
$({\leq}|\alpha|)$--complete, there are conditions $q\in \bbQ$ satisfying
$(*)_\eta+(**)_\eta$. One checks inductively that $\bar{p}^\alpha,
\bar{q}^\alpha \in N$ for all $\alpha<\lambda$ (remember $\st\in N$ and the
choice of ``the $<^*_\chi$--first''). The play $\langle \bar{p}^\alpha,
\bar{q}^\alpha: \alpha<\lambda\rangle$ is won by Generic, so there is a
condition $p^*\geq p$ which is aux-generic over $\bar{q}=\langle
q_{\alpha,\eta}:\alpha<\lambda\ \&\ \eta\in Y_\alpha\rangle$ and $D$. We
claim that $p^*$ is $(N,\bbQ)$--generic. So suppose towards contradiction
that $p^+\geq p^*$, $p^+\forces\name{\tau}_\beta=\zeta$, $\beta<\lambda$ but 
$\zeta\notin N$. Consider a play $\langle (r_\alpha,A_\alpha,\eta_\alpha,
r_\alpha'): \alpha<\lambda \rangle$ of ${\Game^{{\rm aux}}_{\bar{Y}}(p^*,
  \bar{q},\bbQ,\leq,\bar{\leq}_\pr,D)}$ in which COM follows her winning
strategy and INC plays: 
\begin{itemize}
\item $r_0'=p^+$, and for $\alpha>0$ he lets $r_\alpha'=r_\alpha$.
\end{itemize}
Let $\gamma\in \mathop{\triangle}\limits_{\alpha<\lambda} A_\alpha$ be a
limit ordinal greater than $\beta$. Since the play was won by COM, we have
$\eta_\gamma\in Y_\gamma$ and $q_{\gamma,\eta_\gamma}\leq^\gamma_\pr
r_\gamma$. Since $p^+\leq r_\gamma$, we know that $r_\gamma\forces
\name{\tau}_\beta=\zeta$ and hence (by $(**)_{\eta_\gamma}$)
$q_{\gamma,\eta_\gamma}\forces \name{\tau}_\beta=\zeta$. However,
$q_{\gamma,\eta_\gamma}\in N$, contradicting $\zeta\notin N$.  
\end{proof}

\begin{lemma}
  \label{lem1.5}
Assume that $\lambda$ is a regular uncountable cardinal,
$f:\lambda\longrightarrow\lambda+1$ and $\bar{\bbQ}=\langle
\bbP_\xi,\name{\bbQ}_\xi:\xi<\gamma \rangle$ is a $\lambda$--support
iteration such that for every $\xi<\lambda$:
\[\forces_{\bbP_\xi}\mbox{`` $(\name{\bbQ}_\xi,\leq,\bar{\leq}_\pr)$ is a 
  forcing notion with $f$--complete semi-purity ''.}\]
Let  $\cT=(T,\rk)$ be a standard $(w,1)^\gamma$--tree, $w\in
[\gamma]^{<\lambda}$, and let $\bar{p}=\langle p_t: t\in T\rangle$ be a tree
of conditions in $\bbP_\gamma$. Suppose that $\alpha<\lambda$ and $\Upsilon$ 
is a set of $\bbP_\gamma$--names for ordinals such that $|T|\cdot
|\Upsilon|<f(\alpha)$. Then there exists a tree of conditions
$\bar{q}=\langle q_t: t\in T \rangle$ such that  
\begin{enumerate}
\item[$(\circledast)_1$] $\bar{p}\leq \bar{q}$ and  if $t\in T$, $\xi\in
  w\cap \rk(t)$, then $q_t\rest\xi\forces_{\bbP_\xi} p_t(\xi)\leq_\pr^\alpha 
  q_t(\xi)$, and  
\item[$(\circledast)_2$] if $\name{\tau}\in\Upsilon$, $t\in T$,
  $\rk(t)=\gamma$ and there is a condition $q\in \bbP_\gamma$ such that 
\begin{itemize}
\item $q_t\leq q$, and $q\rest\xi\forces_{\bbP_\xi} q_t(\xi)\leq_\pr^\alpha 
  q(\xi)$ for all $\xi\in w$, and
\item $q$ forces a value to $\name{\tau}$, 
\end{itemize}
then $q_t$ forces a value to $\name{\tau}$.
\end{enumerate}
\end{lemma}

\begin{proof}
Let $\kappa=|T|\cdot |\Upsilon|<f(\alpha)$ (and we may assume $\kappa$ is
infinite as otherwise arguments are trivial). Let $\leq^\pr_w$ be a binary
relation on $\bbP_\gamma$ defined by 

$p\leq^\pr_w q$ if and only if 

$p\leq_{\bbP_\gamma} q$ and for each $\xi\in w$,
$q\rest\xi\forces_{\bbP_\xi}p(\xi)\leq_\pr^\alpha q(\xi)$.  

\noindent The relation $\leq^\pr_w$ is extended to trees of conditions in
the natural way. 

For $\xi\in \gamma\setminus w$ let $\name{\st}^0_\xi$ be a
$\bbP_\xi$--name for a wining strategy of Complete in $\Game^{\kappa+1}_0
\big((\name{\bbQ}_\xi,\leq),\name{\emptyset}_{\name{\bbQ}_\xi}\big)$ such
that it instructs her to play $\name{\emptyset}_{\name{\bbQ}_\xi}$ as long
as Incomplete plays $\name{\emptyset}_{\name{\bbQ}_\xi}$. For $\xi\in w$ let 
$\name{\st}^1_\xi$ be a name for a similar strategy for the game
$\Game^{\kappa+1}_0\big((\name{\bbQ}_\xi,\leq_\pr^\alpha),
\name{\emptyset}_{\name{\bbQ}_\xi}\big)$. 

Let $\langle (t_i,\name{\tau}_i):i<\kappa\rangle$ list all members of
$\{t\in T:\rk(t)=\gamma\}\times\Upsilon$ (with possible repetitions). By
induction on $i\leq\kappa$ we choose trees of conditions $\bar{q}^i= 
\langle q^i_t:t\in T\rangle$ and $\bar{r}^i=\langle r^i_t:t\in T\rangle$
such that  
\begin{enumerate}
\item[$(\alpha)$] $\bar{p}\leq^\pr_w \bar{q}^0$, $\bar{q}^i\leq^\pr_w
\bar{r}^i \leq^\pr_w \bar{q}^j\leq^\pr_w\bar{r}^j$ for $i<j\leq\kappa$, 
\item[$(\beta)$] for each $t\in T$, $j\leq\kappa$ and $\xi\in
  \rk(t)\setminus w$, 
\[\begin{array}{ll}
q^j_t\rest\xi\forces_{\bbP_\xi}&\mbox{`` the sequence }\langle (q^i_t(\xi),
r^i_t(\xi): i\leq j\rangle\mbox{ is a legal partial play of }\\  
&\quad\Game_0^{\kappa+1}\big((\name{\bbQ}_\xi,\leq),
\name{\emptyset}_{\name{\bbQ}_\xi}\big)\mbox{ in which Complete follows
}\name{\st}^0_\xi\mbox{ '',}
\end{array}\]  
\item[$(\gamma)$] for each $t\in T$, $j\leq\kappa$ and $\xi\in \rk(t)\cap
  w$, 
\[\begin{array}{ll}
q^j_t\rest\xi\forces_{\bbP_\xi}&\mbox{`` the sequence }\langle (q^i_t(\xi),
r^i_t(\xi): i\leq j\rangle\mbox{ is a legal partial play of }\\   
&\quad\Game_0^{\kappa+1}\big((\name{\bbQ}_\xi,\leq_\pr^\alpha),
\name{\emptyset}_{\name{\bbQ}_\xi}\big)\mbox{ in which Complete follows
}\name{\st}^1_\xi\mbox{ '',}
\end{array}\]  
\item[$(\delta)$] for each $i<\kappa$, if there is a condition
  $q\in\bbP_\gamma$ such that 
\begin{enumerate}
\item[(a)] $q^i_{t_i}\leq^\pr_w q$, and 
\item[(b)] $q$ forces a value to $\name{\tau}_i$, 
\end{enumerate}
then already $q^i_{t_i}$ forces the value to $\name{\tau}_i$.
\end{enumerate}
So suppose we have defined $\bar{q}^j,\bar{r}^j$ for $j<i$. Stipulating
$\bar{r}^{-1}=\bar{p}$, $t_\kappa=t_0$, and $\name{\tau}_\kappa=
\name{\tau}_0$ we ask if there is a condition $q\in\bbP_\gamma$ such that
$r^j_{t_i}\leq^\pr_w q$ for all $j<i$ which forces a value to
$\name{\tau}_i$. If there are such conditions, let $q^i_{t_i}$ be one of
them. Otherwise let $q^i_{t_i}$ be any $\leq^\pr_w$--bound to $\{r_{t_i}^j:
j<i\}$ (there is such a bound by $(\beta)+(\gamma)$). Then for $t\in 
T\setminus\{t_i\}$ define $q^i_t$ so that letting $s=t\cap t_i$: 
\begin{itemize}
\item if $\xi<\rk(s)$, then $q^i_t(\xi)=q^i_{t_i}(\xi)$,
\item if $\rk(s)\leq \xi<\rk(t)$, $\xi\notin w$, then $q^i_t(\xi)$ is the 
  $<^*_\chi$--first $\bbP_\xi$--name such that
\[q^i_t\rest\xi\forces_{\bbP_\xi}\mbox{`` $q^i_t(\xi)$ is a $\leq$--upper
  bound to $\{r^j_t(\xi):j<i\}$ ''},\]
\item if $\rk(s)\leq \xi<\rk(t)$, $\xi\in w$, then $q^i_t(\xi)$ is the
  $<^*_\chi$--first $\bbP_\xi$--name such that
\[q^i_t\rest\xi\forces_{\bbP_\xi}\mbox{`` $q^i_t(\xi)$ is a
  $\leq_\pr^\alpha$--upper bound to $\{r^j_t(\xi):j<i\}$ ''}.\]
\end{itemize}
It should be clear that the above demands correctly define a tree of
conditions $\bar{q}^i=\langle q^i_t:t\in T\rangle$ (note the choice of ``the
$<^*_\chi$--first names''). Finally, we choose $\bar{r}^i$ so that (the
respective instances of) conditions $(\beta)+(\gamma)$ are satisfied. To
ensure we end up with a tree of conditions, at each coordinate we choose
``the $<^*_\chi$--first names for the answers given by the respective
strategies''.

After the inductive process is completed, put $\bar{q}=\bar{q}^\kappa$. 
\end{proof}

\begin{theorem}
  \label{thm1.6}
Assume that $\lambda$ is a strongly inaccessible cardinal,
$f:\lambda\longrightarrow \lambda+1$ and $\bar{\kappa}=\langle
\kappa_\alpha:\alpha<\lambda\rangle$ is a sequence of infinite cardinals
such that $(\kappa_\alpha)^{|\alpha|}<f(\alpha)$ for all $\alpha<\lambda$,
and suppose also that $D$ is a normal filter on $\lambda$. For 
$\xi<\gamma$ let $\bar{Y}^\xi=\langle Y_\alpha^\xi:\alpha<\lambda\rangle$ be
an indexing sequence such that $|Y_\alpha^\xi|\leq\kappa_\alpha$. Let
$\bar{\bbQ}=\langle\bbP_\xi,\name{\bbQ}_\xi:\xi<\gamma\rangle$ be a
$\lambda$--support iteration such that  
\[\forces_{\bbP_\xi}\mbox{`` $\name{\bbQ}_\xi$ is $f$--semi-purely proper over
  $\bar{Y}^\xi,D^{\bV^{\bbP_\xi}}$ ''}\]
for every $\xi<\gamma$ (where $D^{\bV^{\bbP_\xi}}$ is the normal filter on
$\lambda$ generated in $\bV^{\bbP_\xi}$ by $D$).\\  
Then $\bbP_\gamma=\lim(\bar{\bbQ})$ is $\lambda$--proper in the standard
sense. 
\end{theorem}

\begin{proof}
The proof is very similar to that of \cite[Theorem 2.7]{RoSh:888}.

Abusing our notation, the names for the forcing relation and a witness for
the semi-pure properness of $\name{\bbQ}_\xi$ will be denoted $\leq$ and
$\bar{\leq}_\pr=\langle\leq_\pr^\alpha:\alpha<\lambda\rangle$,
respectively. For each $\xi<\gamma$ let $\name{\st}^0_\xi$ be 
the $<^*_\chi$--first $\bbP_\xi$--name for a winning strategy of Complete
in $\Game_0^\lambda(\name{\bbQ}_\xi,\name{\emptyset}_{\name{\bbQ}_\xi})$ such
that it instructs Complete to play $\name{\emptyset}_{\name{\bbQ}_\xi}$ as
long as her opponent plays $\name{\emptyset}_{\name{\bbQ}_\xi}$.   

Let $N\prec (\cH(\chi),\in,<^*_\chi)$ be such that ${}^{<\lambda}N\subseteq
N$, $|N|=\lambda$ and $\bar{\bbQ},D,\langle\bar{Y}^\xi,(\dbQ_\xi,
\leq,\bar{\leq}_\pr):\xi<\gamma\rangle,\ldots\in N$. Let $p\in N\cap
\bbP_\gamma$ and let $\langle \name{\tau}_\alpha:\alpha<\lambda\rangle$ list
all $\bbP_\gamma$--names for ordinals from $N$. Note that if $\xi\in\gamma
\cap N$, then $\name{\st}^0_\xi\in N$.  

By induction on $\delta<\lambda$ we will choose
\begin{enumerate}
\item[$(\otimes)_\delta$] $\cT_\delta,w_\delta,r_\delta^-,r_\delta,
  \bar{p}^\delta_*, \bar{q}^\delta_*$ and $\name{\bar{p}}_{\delta,\xi},
  \name{\bar{q}}_{\delta,\xi},\name{\st}_\xi$ for $\xi\in N\cap \gamma$
\end{enumerate}
so that the following demands are satisfied. 

\begin{enumerate}
\item[$(*)_0$] All objects listed in $(\otimes)_\delta$ belong to $N$. After
  stage $\delta<\lambda$ of the construction, these objects are known for
  $\delta$ and $\xi\in w_\delta$.
\item[$(*)_1$] $r_\delta^-,r_\delta\in\bbP_\gamma$, $r^-_0(0)=r_0(0)=p(0)$,
  $w_\delta\subseteq \gamma$, $|w_\delta|=|\delta+1|$, $w_0=\{0\}$,
  $w_\delta\subseteq w_{\delta+1}$, and if $\delta$ is limit then $w_\delta
  =\bigcup\limits_{\alpha<\delta} w_\alpha$, and
\[\bigcup\limits_{\alpha<\lambda} \dom(r_\alpha)=
\bigcup\limits_{\alpha<\lambda} w_\alpha=N\cap\gamma.\] 
\item[$(*)_2$] For each $\alpha<\delta<\lambda$ we have $\big(\forall \xi\in
  w_{\alpha+1}\big)\big(r_\alpha(\xi)=r_\delta(\xi)\big)$ and $p\leq
  r^-_\alpha\leq r_\alpha\leq r^-_\delta\leq r_\delta$.
\item[$(*)_3$] If $\xi\in (\gamma\setminus w_\delta)\cap N$, then 
\[\begin{array}{ll}
r_\delta\rest\xi\forces&\mbox{`` the sequence }\langle r^-_\alpha(\xi),
r_\alpha(\xi):\alpha\leq\delta\rangle\mbox{ is a legal partial play of }\\
&\quad\Game_0^\lambda\big(\name{\bbQ}_\xi,
\name{\emptyset}_{\name{\bbQ}_\xi}\big)\mbox{ in which Complete follows
}\name{\st}^0_\xi\mbox{ ''}  
\end{array}\]
and if $\xi\in w_{\delta+1}\setminus w_\delta$, then $\name{\st}_\xi\in N$
is a $\bbP_\xi$--name for a winning strategy of Generic in
$\masgamexi$. (And $\st_0\in N$ is a winning strategy of Generic in
$\Game^{{\rm main}}_{\bar{Y}^0}(p(0),\bbQ_0,\leq,\bar{\leq}_\pr,D)$.)     
\item[$(*)_4$] $\cT_\delta=(T_\delta,\rk_\delta)$ is a standard $(w_\delta,
1)^\gamma$--tree, $T_\delta=\bigcup\limits_{\alpha\leq\gamma}
\prod\limits_{\xi\in w_\delta\cap\alpha} Y_\delta^\xi$ (so $T_\delta$
consists of all sequences $\bar{t}=\langle t_\xi:\xi\in w_\delta\cap \alpha
\rangle$ where $\alpha\leq\gamma$ and $t_\xi\in Y^\xi_\delta$).
\item[$(*)_5$] $\bar{p}^\delta_*=\langle p^\delta_{*,t}:t\in T_\delta
\rangle$ and $\bar{q}^\delta_*=\langle q^\delta_{*,t}:t\in T_\delta\rangle$
are standard trees of conditions, $\bar{p}^\delta_*\leq\bar{q}^\delta_*$.  
\item[$(*)_6$] For $t\in T_\delta$ we have that $\dom(p^\delta_{*,t})=
  \big(\dom(p)\cup\bigcup\limits_{\alpha<\delta}\dom(r_\alpha)\cup
  w_\delta\big) \cap \rk_\delta(t)$ and for each $\xi\in
  \dom(p^\delta_{*,t})\setminus w_\delta$: 
\[\begin{array}{ll}
p^\delta_{*,t}\rest\xi\forces_{\bbP_\xi}&\mbox{`` if the set } \{
r_\alpha(\xi):\alpha<\delta\}\cup\{p(\xi)\}\mbox{ has an upper bound in
}\name{\bbQ}_\xi,\\  
&\mbox{\quad then $p^\delta_{*,t}(\xi)$ is such an upper bound ''.}
  \end{array}\]
\item[$(*)_7$] For $\xi\in N\cap\gamma$, $\name{\bar{p}}_{\delta,\xi} =\langle
\name{p}_{\delta,\eta}^\xi:\eta\in Y_\delta^\xi\rangle$ and
$\name{\bar{q}}_{\delta,\xi}=\langle \name{q}_{\delta,\eta}^\xi: \eta\in
Y_\delta^\xi\rangle$ are $\bbP_\xi$--names for systems of conditions in
$\name{\bbQ}_\xi$ indexed by $Y_\delta^\xi$. 
\item[$(*)_8$] If $\xi\in w_{\beta+1}\setminus w_\beta$, $\beta<\lambda$,
  then 
\[\begin{array}{r}
\forces_{\bbP_\xi}\mbox{`` }\langle\name{\bar{p}}_{\alpha,\xi},
\name{\bar{q}}_{\alpha,\xi}:\alpha<\lambda \rangle\mbox{ is a play of
}\masgamezero(r_\beta(\xi),\name{\bbQ}_\xi,\leq,\bar{\leq}_\pr,
D^{\bV^{\bbP_\xi}})\\    
\mbox{ in which Generic uses $\name{\st}_\xi$ ''.}
  \end{array}\]
\item[$(*)_9$] If $t\in T_\delta$, $\rk_\delta(t)=\xi<\gamma$, then for
  each $\eta\in Y_\delta^\xi$
\[q^\delta_{*,t}\forces_{\bbP_\xi}\mbox{`` } \name{p}^\xi_{\delta,\eta}
=p^\delta_{*,t\cup\{\langle\xi,\eta\rangle\}}(\xi)\mbox{ and } 
\name{q}^\xi_{\delta,\eta}=q^\delta_{*,t\cup\{\langle\xi,\eta\rangle\}}(\xi)  
\mbox{ ''.}\]  
\item[$(*)_{10}$] If $t\in T_\delta$, $\rk_\delta(t)=\gamma$ and $\alpha< 
  \delta$ and there is a condition $q\in\bbP_\gamma$ such that
\begin{enumerate}
\item[(a)] $q^\delta_{*,t}\leq q$, and 
\item[(b)] $q\rest\xi\forces_{\bbP_\xi} q^\delta_{*,t}(\xi)\leq_\pr^\delta q(\xi)$
  for all $\xi\in w_\delta$ and
\item[(c)] $q$ forces a value to $\name{\tau}_\alpha$, 
\end{enumerate}
then already the condition $q^\delta_{*,t}$ forces the value to 
$\name{\tau}_\alpha$. 
\item[$(*)_{11}$] $\dom(r_\delta^-)=\dom(r_\delta)= \bigcup\limits_{t\in
    T_\delta} \dom(q^\delta_{*,t})$ and if $t\in T_\delta$, $\xi\in
  \dom(r_\delta)\cap \rk_\delta(t)\setminus w_\delta$, and $q^\delta_{*,t}
  \rest\xi\leq q\in\bbP_\xi$, $r_\delta\rest\xi\leq q$, then   
\[\begin{array}{ll}
q\forces_{\bbP_\xi}&\mbox{`` if the set }\{r_\alpha(\xi):\alpha<\delta\}
\cup\{q^\delta_{*,t}(\xi), p(\xi)\}\mbox{ has an upper bound in }  
\name{\bbQ}_\xi,\\ 
&\mbox{\quad then $r_\delta^-(\xi)$ is such an upper bound ''.}
  \end{array}\]
\end{enumerate}

We start with fixing an increasing continuous sequence $\langle w_\alpha:
\alpha<\lambda\rangle$ of subsets of $N\cap\gamma$ such that the demands of
$(*)_1$ are satisfied. Now, by induction on $\delta<\lambda$ we choose the
other objects. So assume that we have defined all objects listed in
$(\otimes)_\alpha$ for $\alpha<\delta$.

To ensure $(*)_0$, whenever we say ``choose an $X$ such that$\ldots$'' we
mean ``choose the $<^*_\chi$--first $X$ such that$\ldots$''. This convention
will guarantee that our choices are from $N$.

If $\delta$ is a successor ordinal and $\xi\in w_\delta\setminus
w_{\delta-1}$, then let $\name{\st}_\xi\in N$ be a $\bbP_\xi$--name for a
winning strategy of Generic in $\masgamezero(r_{\delta-1}(\xi),
\name{\bbQ}_\xi,\leq,\bar{\leq}_\pr,D^{\bV^{\bbP_\xi}})$. We also pick
$\name{\bar{p}}_{\alpha,\xi},\name{\bar{q}}_{\alpha,\xi}$ for $\alpha<
\delta$ so that $(*)_7+(*)_8$ hold (note that we already know $r_{\delta-1}
(\xi)$ and by $(*)_2$ it is going to be equal to $r_\delta(\xi)$).     

Clause $(*)_4$ fully describes $\cT_\delta$. Note that, by the assumptions
on $\bar{Y},\bar{\kappa}$,
\begin{enumerate}
\item[$(*)_{12}$] $|T_\delta|\leq (\kappa_\delta)^{|\delta|}
  <f(\delta)$ so also $|T_\delta|\cdot |\delta|<f(\delta)$. 
\end{enumerate}
For each $\xi\in w_\delta$ we choose a $\bbP_\xi$--name
$\name{\bar{p}}_{\delta,\xi}$ such that  
\[\begin{array}{ll}
\forces_{\bbP_\xi}&\mbox{`` }\name{\bar{p}}_{\delta,\xi}=\langle
\name{p}_{\delta,\eta}^\xi:\eta\in Y_\delta^\xi\rangle\mbox{ is given to  
Generic by }\name{\st}_\xi\mbox{ as an answer to}\\ 
&\quad\langle\name{\bar{p}}_{\alpha,\xi},\name{\bar{q}}_{\alpha,\xi}:\alpha<
\delta\rangle\mbox{ in the game }\Game^{\rm main}_{\bar{Y}^\xi}(
r_\beta(\xi), \name{\bbQ}_\xi,\leq,\bar{\leq}_\pr, D^{\bV^{\bbP_\xi}}),
\mbox{ ''} 
\end{array}\]
where $\beta<\delta$ is such that $\xi\in w_{\beta+1}\setminus
w_\beta$. (Note that for each $\xi\in w_\delta$ and distinct
$\eta_0,\eta_1\in Y_\delta^\xi$  we have $\forces_{\bbP_\xi}$`` the
conditions $\name{p}^\xi_{\delta,\eta_0}, \name{p}^\xi_{\delta,\eta_1}$ are  
incompatible''.) Next we choose a tree of conditions $\bar{p}^\delta_*= 
\langle p^\delta_{*,t}: t\in T_\delta\rangle$ such that for each $t\in
T_\delta$: 
\begin{itemize}
\item $\dom(p^\delta_{*,t})=\big(\dom(p)\cup \bigcup\limits_{\alpha<\delta} 
  \dom(r_\alpha) \cup w_\delta\big) \cap \rk_\delta(t)$ and 
\item for $\xi\in\dom(p^\delta_{*,t})\setminus w_\delta$,
  $p^\delta_{*,t}(\xi)$ is the $<^*_\chi$--first $\bbP_\xi$--name for a
  condition in $\name{\bbQ}_\xi$ such that  
\[\begin{array}{ll}
p^\delta_{*,t}\rest\xi\forces_{\bbP_\xi}&\mbox{`` if the set } \{
r_\alpha(\xi):\alpha<\delta\}\cup\{p(\xi)\}\mbox{ has an upper bound in
}\name{\bbQ}_\xi,\\  
&\mbox{\quad then $p^\delta_{*,t}(\xi)$ is such an upper bound '',}
\end{array}\]
\item $p^\delta_{*,t}(\xi)=\name{\bar{p}}^\xi_{\delta,(t)_\xi}$ for 
  $\xi\in \dom(p^\delta_{*,t})\cap w_\delta$. 
\end{itemize}
Because of $(*)_{12}$ we may use Lemma \ref{lem1.5} to pick a tree of
conditions $\bar{q}^\delta_*=\langle q^\delta_{*,t}:t\in T_\delta\rangle$
such that  
\begin{itemize}
\item $\bar{p}^\delta_*\leq\bar{q}^\delta_*$, 
\item if $t\in T_\delta$, $\xi\in w_\delta\cap \rk_\delta(t)$, then 
  $q^\delta_{*,t}\rest \xi\forces_{\bbP_\xi} p^\delta_{*,t}(\xi)
  \leq_{\pr}^\delta q^\delta_{*,t}(\xi)$,
\item if $t\in T_\delta$, $\rk_\delta(t)=\gamma$ and $\alpha<\delta$ and
  there is a condition $q\in \bbP_\gamma$ such that
\begin{enumerate}
\item[(a)] $q^\delta_{*,t}\leq q$, and
\item[(b)] $q\rest\xi\forces_{\bbP_\xi} q^\delta_{*,t}(\xi)\leq_\pr^\delta
  q(\xi)$ for all $\xi\in w_\delta$ and
\item[(c)] $q$ forces a value to $\name{\tau}_\alpha$, 
\end{enumerate}
then $q^\delta_{*,t}$ forces a value to $\name{\tau}_\alpha$. 
\end{itemize}
Note that if $\xi\in w_\delta$, $t\in T_\delta$, $\rk_\delta(t)=\xi$ and
$\eta_0,\eta_1\in Y_\delta^\xi$ are distinct, then 
\[q^\delta_{*,t}\forces_{\bbP_\xi}\mbox{`` the conditions
$q^\delta_{*,t\cup\{\langle\xi,\eta_0\rangle\}}(\xi),
q^\delta_{*,t\cup\{\langle\xi,\eta_1\rangle\}}(\xi)$ are incompatible ''.}\] 
Therefore we may choose $\bbP_\xi$--names $\name{q}^\xi_{\delta,\eta}$ (for
$\xi\in w_\delta$) such that
\begin{itemize}
\item $\forces_{\bbP_\xi}\mbox{``}\name{\bar{q}}_{\delta,\xi}=\langle
  \name{q}^\xi_{\delta,\eta}:\eta\in Y_\delta^\xi\rangle$ is a system of
    conditions in $\name{\bbQ}_\xi$ indexed by $Y_\delta^\xi$'',
\item $\forces_{\bbP_\xi}\mbox{`` }(\forall \eta\in Y_\delta^\xi)( 
\name{p}^\xi_{\delta,\eta} \leq_\pr^\delta \name{q}^\xi_{\delta,\eta})$ '',  
\item if $t\in T_\delta$, $\rk_\delta(t)>\xi$, then $q^\delta_{*,t\rest \xi}
  \forces_{\bbP_\xi} q^\delta_{*,t}(\xi)=\name{q}^\xi_{\delta,(t)_\xi}$. 
\end{itemize}
Finally, we define $r^-_\delta,r_\delta\in\bbP_\gamma$ so that
\[\dom(r^-_\delta)=\dom(r_\delta)=\bigcup\limits_{t\in T_\delta}\dom(
q^\delta_{*,t})\]
 and 
\begin{itemize}
\item $r_0^-(0)=r_0(0)=p(0)$,
\item if $\xi\in w_{\alpha+1}$, $\alpha<\delta$, then $r^-_\delta(\xi) =
  r_\delta(\xi)=r_\alpha(\xi)$, 
\item if $\xi\in\dom(r^-_\delta)\setminus w_\delta$, then $r^-_\delta(\xi)$
  is the $<^*_\chi$--first $\bbP_\xi$--name for an element of
  $\name{\bbQ}_\xi$ such that   
\[\begin{array}{ll}
r^-_\delta\rest\xi\forces_{\bbP_\xi}&\mbox{`` }r^-_\delta(\xi)\mbox{ is an
upper bound of }\{r_\alpha(\xi):\alpha<\delta\}\cup\{p(\xi)\}\mbox{ and }\\  
&\mbox{\quad if }t\in T_\delta,\ \ \rk_\delta(t)>\xi,\mbox{ and }
q^\delta_{*,t}\rest\xi\in\name{G}_{\bbP_\xi}\mbox{ and the set}\\
&\quad\{r_\alpha(\xi):\alpha<\delta\}\cup\{q^\delta_{*,t}(\xi),
  p(\xi)\}\mbox{ has an upper bound in }\name{\bbQ}_\xi,\\ 
&\mbox{\quad then $r^-_\delta(\xi)$ is such an upper bound '',}
  \end{array}\]
and $r_\delta(\xi)$ is the $<^*_\chi$--first $\bbP_\xi$--name for an element 
  of $\name{\bbQ}_\xi$ such that 
\[\begin{array}{ll}
r_\delta\rest\xi\forces_{\bbP_\xi}&\mbox{`` }r_\delta(\xi)\mbox{ is given to
  Complete by }\name{\st}^0_\xi\mbox{ as the answer to }\ \\ 
&\quad\langle r^-_\alpha(\xi),r_\alpha(\xi):\alpha<\delta\rangle\conc
  \langle r^-_\delta(\xi)\rangle \mbox{ in the game }\Game^\lambda_0(
  \name{\bbQ}_\xi, \name{\emptyset}_{\name{\bbQ}_\xi})\mbox{ ''.}  
  \end{array}\]
\end{itemize}
It follows from $(*)_2+(*)_3$ from the previous stages that
$r_\delta^-,r_\delta$ are well defined and $p,r_\alpha\leq r^-_\delta\leq
r_\delta$ for $\alpha<\delta$ (using induction on $\xi\in \dom(r_\delta)$).  
\medskip

This completes the description of the inductive definition of the objects
listed in $(\otimes)_\delta$; it should be clear from the construction that
demands $(*)_0$--$(*)_{11}$ are satisfied. For each $\xi\in w_{\beta+1}
\setminus w_\beta$, $\beta<\lambda$, look at the sequence $\langle
\name{\bar{p}}_{\delta,\xi},\name{\bar{q}}_{\delta,\xi}:\delta<\lambda
\rangle$ and use $(*)_8$ to choose a $\bbP_\xi$--name $q(\xi)$ for a
condition in $\name{\bbQ}_\xi$ such that  
\[\forces_{\bbP_\xi}\mbox{`` }q(\xi)\geq r_\beta(\xi)\mbox{ is aux-generic 
  over }\langle \name{q}^\xi_{\delta,\eta}:\delta<\lambda\ \&\ \eta\in
Y_\delta^\xi \rangle\mbox{ and } D^{\bV^{\bbP_\xi}}\mbox{ ''}\]
(if $\xi=0$ then $q(0)\geq r_0(0)$ is aux-generic over $\langle
q^0_{\delta,\eta}: \delta<\lambda\ \&\ \eta\in Y_\delta^0\rangle$,
$D$). This determines a condition $q\in\bbP_\gamma$ with $\dom(q)=N\cap
\gamma$. It follows from $(*)_2$ that $p\leq r_\beta\leq q$ for all
$\beta<\lambda$.   
\medskip

Let us argue that $q$ is $(N,\bbP_\gamma)$--generic. Let $\name{\tau}\in N$
be a $\bbP_\gamma$--name for an ordinal, say $\name{\tau}=
\name{\tau}_{\alpha^*}$, $\alpha^*<\lambda$, and let us show that $q\forces
\name{\tau}\in N$. So suppose towards contradiction that $q'\geq q$,
$q'\forces \name{\tau}=\zeta$, $\zeta\notin N$. For each $\xi\in N\cap
\gamma$ fix a $\bbP_\xi$--name $\name{\st}^+_\xi$ such that 
\[\begin{array}{ll}
\forces_{\bbP_\xi}&\mbox{`` $\name{\st}^+_\xi$ is a winning strategy of COM 
  in }\\
&\ \ \auxk\big(q(\xi),\langle \name{q}^\xi_{\delta,\eta}:\delta<\lambda\ \&\
\eta\in Y_\delta^\xi\rangle,\name{\bbQ}_\xi,\leq,\bar{\leq}_\pr,
D^{\bV^{\bbP_\xi}} \big)\mbox{ ''.}  
\end{array}\]
Construct inductively a sequence
\[\langle r^+_\alpha,r'_\alpha,\eta_\alpha(\xi),\name{\eta}_\alpha(\xi),
\langle A^i_\alpha(\xi),\name{A}^i_\alpha(\xi): i<\lambda \rangle,
\name{A}_\alpha(\xi):\alpha<\lambda\ \&\ \xi\in N\cap\gamma\rangle\]   
such that the following demands $(*)_{13}$--$(*)_{15}$ are satisfied.
\begin{enumerate}
\item[$(*)_{13}$] $r^+_\alpha,r'_\alpha\in\bbP_\gamma$, $r^+_0=q$, $r'_0\geq
q'$ and $r^+_\beta\leq r'_\beta\leq r^+_\alpha$ for $\beta<\alpha<\lambda$.  
\item[$(*)_{14}$] For each $\xi\in N\cap\gamma$ and $\alpha<\lambda$ we have
  that $\eta_\alpha(\xi)\in {}^\alpha\lambda$, $A^i_\alpha(\xi)\in D$,
  $\name{\eta}_\alpha(\xi)$ is a $\bbP_\xi$--name for a member of
  ${}^\alpha\lambda$, $\name{A}^i_\alpha(\xi)$ is a $\bbP_\xi$--name for a
  member of $D$ and $\name{A}_\alpha(\xi)$ is a $\bbP_\xi$--name for a 
  member of $D^{\bV^{\bbP_\xi}}$, and 
\[\begin{array}{ll}
\forces_{\bbP_\xi}&\mbox{`` }\langle (r_\alpha^+(\xi),
\name{A}_\alpha(\xi), \name{\eta}_\alpha(\xi),r'_\alpha(\xi)): \alpha<
\lambda\rangle\mbox{ is a result of a play of}\\
\ \ &\auxk\big(q(\xi),\langle \name{q}^\xi_{\delta,\eta}:\delta<\lambda\
\&\ \eta\in Y_\delta^\xi\rangle,\name{\bbQ}_\xi,\leq,\bar{\leq}_\pr,
D^{\bV^{\bbP_\xi}}\big)\\  
\ \ &\mbox{in which COM follows the strategy $\name{\st}^+_\xi$ ''.}
\end{array}\]
\item[$(*)_{15}$] For $j,\beta\leq\alpha<\lambda$ and $\xi\in w_\alpha$ we
  have  
\[r'_\alpha\rest \xi\forces\mbox{`` }\name{\eta}_\alpha(\xi)=
\eta_\alpha(\xi)\ \ \&\ \ \mathop{\triangle}\limits_{i<\lambda}
\name{A}^i_\alpha(\xi) \subseteq\name{A}_\alpha(\xi)\ \ \&\ \ 
\name{A}^j_\beta(\xi)= A^j_\beta(\xi)\mbox{ ''.}\]  
\end{enumerate}
(It should be clear how to carry out the construction; remember
$\bbP_\gamma$ is $({<}\lambda)$--strategically complete, so in particular it
does not add new members of ${}^\alpha\lambda$ for $\alpha<\lambda$.) Take a
limit ordinal $\vare>\alpha^*$ such that $\vare\in\bigcap\limits_{\xi\in
w_\vare} \bigcap\limits_{i,j<\vare} A^i_j(\xi)$. Then, by
$(*)_{13}$--$(*)_{15}$, for each $\xi\in w_\vare$ we have     
\[r^+_\vare\rest\xi\forces_{\bbP_\xi}\mbox{`` }\vare\in
\mathop{\triangle}_{\alpha<\lambda}\name{A}_\alpha(\xi)\ \mbox{ and }\
\name{\eta}_\vare(\xi)= \bigcup\limits_{\alpha<\vare}\eta_\alpha(\xi)=
\eta_\vare(\xi)\in Y^\xi_\vare\mbox{ ''}\]    
and consequently, by $(*)_{14}$, 
\begin{enumerate}
\item[$(*)_{16}$] $r^+_\vare\rest\xi\forces_{\bbP_\xi}$``
$\name{q}^\xi_{\vare,\eta_\vare(\xi)}\leq_\pr^\vare r^+_\vare(\xi)$ '' for
each $\xi\in w_\vare$.
\end{enumerate}
Also note that 
\begin{enumerate}
\item[$(*)_{17}$] $p\leq r_\delta\leq q\leq r^+_\vare$ for all
  $\delta<\lambda$.  
\end{enumerate}
Let $t\in T_\vare$ be such that $\rk_\vare(t)=\gamma$ and
$(t)_\xi=\eta_\vare(\xi)$ for $\xi\in w_\vare$. By induction on
$\xi\leq\gamma$, $\xi\in N$, we show that
$q^\vare_{*,t}\rest\xi\leq_{\bbP_\xi} r^+_\vare\rest\xi$. So let us assume
that $\xi<\gamma$ and we have shown that $q^\vare_{*,t}\rest \xi
\leq_{\bbP_\xi} r^+_\vare\rest\xi$. If $\xi\in w_\vare$ then by
$(*)_9+(*)_{16}$ we have $q^\vare_{*,t}\rest(\xi+1)\leq_{\bbP_{\xi+1}}
r^+_\vare\rest (\xi+1)$. So assume $\xi\notin w_\vare$. Now, by $(*)_5$,
$p^\vare_{*,t}\rest\xi\leq r^+_\vare\rest\xi$, so 
\[r^+_\vare\rest \xi\forces_{\bbP_\xi}\mbox{`` }r_\alpha(\xi)\leq 
p^\vare_{*,t}(\xi)\mbox{ for all }\alpha<\vare\mbox{ ''}\]
(remember $(*)_{17}+(*)_6$), and hence 
\[r^+_\vare\rest\xi\forces_{\bbP_\xi}\mbox{`` }r_\alpha(\xi)\leq
q^\vare_{*,t}(\xi)\mbox{ for all }\alpha<\vare\mbox{ ''.}\]
Consequently, it follows from $(*)_{11}$ that 
\[r^+_\vare\rest\xi\forces_{\bbP_\xi}\mbox{`` }q^\vare_{*,t}(\xi)\leq
r^-_\vare(\xi)\leq r_\vare(\xi)\leq r^+_\vare(\xi)\mbox{ ''}\] 
and thus $q^\vare_{*,t}\rest(\xi+1)\leq_{\bbP_{\xi+1}} r^+_\vare\rest
(\xi+1)$. 

Now, since $q^\vare_{*,t}\leq r^+_\vare$ and $(*)_{16}$ holds, we may use
the condition $(*)_{10}$ to conclude that $q^\vare_{*,t}
\forces_{\bbP_\gamma} \name{\tau}=\zeta$ (remember $q'\leq r^+_\vare$,
$\alpha^*<\vare$) and consequently $\zeta\in N$, a contradiction. 
\end{proof}

\begin{remark}
Semi--pure properness is very similar to being reasonably merry of
\cite[Section 6]{RoSh:888}. Despite of some differences in the parameters 
involved, one may suspect that the games are essentially the same if 
$\leq^\delta_\pr = \leq$. This would suggest that semi--pure properness is a 
weaker condition than being reasonably merry. However, the index sets
$Y_\delta$ here are known {\em before\/} the master game starts, while in
\cite{RoSh:888}  the index sets $I_\delta$ are decided at the stage $\delta$
of the game. This makes our present notion somewhat stronger. Note that in
our proof of the Iteration Theorem \ref{thm1.6} we really have to know
$Y_\delta$'s in advance -- we cannot decide names for them and take care of
$(*)_8+(*)_9$ at the same time. (This obstacle was not present in the proof
of \cite[Theorem 6.4]{RoSh:888} as there we did not deal with the auxilary 
relations $\leq^\delta_\pr$.) 

It should be noted that some of the $\lambda$--semi-purely proper forcing
notions discussed in the next section (see Proposition \ref{prop2.4}) are
not reasonably merry as they do not have the bounding property of
\cite[Theorem 6.4(b)]{RoSh:888}. 
\end{remark}

\begin{problem}
Are there any relationships between semi--pure properness and the
  properties introduced in \cite[Definition A.3.6]{RoSh:777},
  \cite[Defnitions 2.2, 6.3]{RoSh:888} ? 
\end{problem}

\section{The Forcings}
In this section we will show that our ``last forcing standing''
$\bbQ^2_\lambda$ and some of its relatives fit the framework of semi--pure
properness (so their $\lambda$--support iterations preserve $\lambda^+$). A
slight modification of $\bbQ^2_\lambda$ was used in iterations in Friedman
and Zdomskyy \cite{FrZd10} and Friedman, Honzik and Zdomskyy
\cite{FrHoZdxx}. It was called ${\rm Miller}(\lambda)$ there and the main
difference between the two forcings is in condition \cite[Definition
2.1(vi)]{FrZd10}. 

The filter $D$ from the previous section will be the club filter, so it is
not mentioned; also until Proposition \ref{stuffFilter} the auxiliary
relations $\leq^\alpha_\pr$ do not depend on $\alpha$, so instead of
$\bar{\leq}_\pr$ we have just $\leq_\pr$ and $f(\alpha)=\lambda$ so we write
$\lambda$ instead of $f$ (see Definition \ref{def1.3}(6)).

For our results we have to assume that $\lambda$ is strongly inaccessible;
the case of successor $\lambda$ remains untreated here (we will deal with it
in a subsequent paper).

\begin{definition}
\label{def2.1}
\begin{enumerate}
\item Let $\bT^{\rm club}$ be the family of all complete $\lambda$--trees $T\subseteq
  {}^{<\lambda}\lambda$ such that 
\begin{itemize}
\item if $t\in T$, then $|\suc_T(t)|=1$ or $\suc_T(t)$ is a club of
  $\lambda$, and  
\item $(\forall t\in T)(\exists s\in T)(t\vtl s\ \&\ |\suc_T(s)|>1)$.   
\end{itemize}
\item We define a forcing notion $\qone$ as follows:\\
{\bf a condition} in $\qone$ is a tree $T\in \bT^{\rm club}$ such that 
\begin{itemize}
\item if $\langle t_i:i<j\rangle\subseteq T$ is $\vtl$--increasing,
  $|\suc_T(t_i)|>1$ for all $i<j$ and $t=\bigcup\limits_{i<j}t_i$, then ($t\in
  T$ and) $|\suc_T(t)|>1$,
\end{itemize}
{\bf the order} $\leq$ of $\qone$ is the inverse inclusion, i.e., $T_1\leq
T_2$ if and only if $T_2\subseteq T_1$.
\item Forcings notions $\qtwo, \qthree,\qfour$ are defined analogously, but\\
{\bf a condition} in $\qtwo$ is a tree $T\in \bT^{\rm club}$ such that for
every $\lambda$--branch $\eta\in\lim_\lambda(T)$ the set $\{\alpha\in
\lambda: |\suc_T(\eta\rest\alpha)|>1\}$ contains a club of $\lambda$,\\ 
{\bf a condition} in $\qthree$ is a tree $T\in \bT^{\rm club}$ such that for
some club $C\subseteq \lambda$ we have 
\[(\forall t\in T)(\lh(t)\in C\ \Rightarrow\ |\suc_T(t)|>1),\]
{\bf a condition} in $\qfour$ is a tree $T\in \bT^{\rm club}$ such that 
\[(\forall t\in T)(\mrot(T)\vtl t\ \Rightarrow\ |\suc_T(t)|>1).\]
\item For $\ell=1,2,3,4$ we define a binary relation $\leq_\pr$ on $\qell$
  by\\ 
$T_1\leq_\pr T_2$ if and only if $T_1\leq T_2$ and $\mrot(T_1)=\mrot(T_2)$.  
\item Let $\bbQ^{1,*}_\lambda$ consists of all conditions $T\in \qtwo $ such 
  that for each $\lambda$--branch $\eta\in\lim_\lambda(T)$ the set
  $\{\alpha\in\lambda: |\suc_T(\eta\rest\alpha)|>1\}$ {\bf is} a club of 
  $\lambda$. 
\item Let $\bbQ^{3,*}_\lambda$ consists of all conditions $T\in \qthree$ such
  that for some club $C\subseteq \lambda$ we have
\begin{itemize}
\item if $t\in T$ and $\lh(t)\in C$, then $|\suc_T(t)|>1$, and 
\item if $t\in T$ and $\lh(t)\notin C$, then $|\suc_T(t)|=1$.
\end{itemize}
\end{enumerate}
\end{definition}

\begin{observation}
\label{obs2.4}
Let $T\in\bT^{\rm club}$. Then $T\in \bbQ^{1,*}_\lambda$ if and only if
there exists a sequence $\langle F_\alpha:\alpha<\lambda\rangle$ of fronts
of $T$ such that
\begin{itemize}
\item if $\alpha<\beta<\lambda$, $t\in F_\beta$, then there is $s\in
  F_\alpha$ such that $s\vtl t$,
\item if $\delta<\lambda$ is limit, $t_\alpha\in F_\alpha$ (for
  $\alpha<\delta$) are such that $t_\alpha\vtl t_\beta$ whenever
  $\alpha<\beta<\delta$, then $\bigcup\limits_{\alpha<\delta} t_\alpha\in
  F_\delta$,
\item for each $t\in T$, $|\suc_T(t)|>1$ if and only if $t\in
  \bigcup\limits_{\alpha<\lambda} F_\alpha$.
\end{itemize}
\end{observation}

\begin{observation}
\label{prop2.3}  
$\qfour\subseteq\bbQ^{3,*}_\lambda\subseteq\bbQ^2_\lambda=\bbQ^{1,*}_\lambda
\subseteq \qtwo$ and $\bbQ^{3,*}_\lambda\subseteq \qthree\subseteq \qtwo$,
and $\bbQ^{3,*}_\lambda$ is a dense subforcing of $\qthree$.  
\end{observation}

\begin{observation}
\label{prop2.2}
Let $\ell\in\{1,2,3,4\}$.
\begin{enumerate}
\item $(\qell,\leq,\leq_\pr)$ is a forcing notion with $\lambda$--complete
  semi-purity.   
\item Moreover, the relations $(\qell,\leq)$ and $(\qell,\leq_\pr)$ are
  $(<\lambda)$--complete.   
\end{enumerate}
\end{observation}

\begin{lemma}
\label{pre2.5}
Let $0<\ell\leq 4$. Assume that $T^\delta\in\qell$ and $F_\delta\subseteq
T^\delta$ (for $\delta<\lambda$) are such that 
\begin{enumerate}
\item[(i)] $F_\delta$ is a front of $T^\delta$, $T^{\delta+1}\subseteq
  T^\delta$, and $F_\delta\subseteq T^{\delta+1}$, 
\item[(ii)] if $\delta$ is limit, then $T^\delta=\bigcap\limits_{i<\delta}
  T^i$ and $F_\delta=\big\{t\in T^\delta: (\forall \xi<\delta)(\exists
  i<\lh(t))(t\rest i\in F_\xi)$ and $(\forall i<\lh(t))(\exists
  \xi<\delta)(\exists j<\lh(\nu))(i<j\ \&\ \nu\rest j\in F_\xi)\big\}$,  
\item[(iii)] $(\forall t\in F_{\delta+1})(\exists s\in F_\delta)(s\vtl t)$,  
\item[(iv)] if $t\in F_\delta$ and $|\suc_{T^\delta}(t)|>1$, then
  $|\suc_{T^{\delta+1}}(t)|>1$.  
\end{enumerate}
Then $S\stackrel{\rm def}{=}\bigcap\limits_{\delta<\lambda}
T^\delta\in\qell$. 
\end{lemma}

\begin{proof}
Plainly, $S$ is a tree closed under unions of $\vtl$--chains shorter than
$\lambda$, and by (i)--(iii) we see that for each $t\in S$ there is $s\in S$
such that $t\vtl s$. Hence $S$ is a complete $\lambda$--tree. 

Also, for each $\alpha<\lambda$ we have 
\begin{enumerate}
\item[(v)] $F_\alpha$ is a front of $S$ and for all $\beta\geq \alpha$
\[\{t\in S:(\exists s\in
F_\alpha)(t\trianglelefteq s)\}=\{t\in T_\beta:(\exists s\in
F_\alpha)(t\trianglelefteq s)\}.\] 
\end{enumerate}
Hence every splitting node in $S$ splits into a club. Suppose now that 
$s\in S$ and let $\eta\in\lim_\lambda(S)$ be such that $s\vtl\eta$. Since
$T^i\in\qtwo$ (remember \ref{prop2.3}), the set $\{\alpha<\lambda:
|\suc_{T^i}(\eta\rest\alpha)|>1\}$ contains a club (for each
$i<\lambda$). Also the set $\{\alpha<\lambda: \eta\rest\alpha\in F_\alpha\}$
is a club (remember (iii)+(ii)). So we may pick a limit ordinal
$\delta<\lambda$ such that $\lh(s)<\delta$, $\eta\rest \delta\in F_\delta$
and $|\suc_{T^i}(\eta\rest\delta)|>1$ for all $i<\delta$. Then (by (ii))
also $|\suc_{T^\delta}(\eta\rest\delta)|>1$ and hence (by (iv)+(iii)+(v)) 
$|\suc_S(\eta\rest\delta)|>1$ (and $s\vtl\eta\rest\delta$). So we may
conclude that $S\in\bT^{\rm club}$. We will argue that $S\in\qell$ 
considering the four possible values of $\ell$ separately.

\noindent {\sc Case}\quad $\ell=1$\\
Suppose $\eta\in\lim_\lambda(S)$. Then for each $\delta<\lambda$ the set
$\{\alpha<\lambda: |\suc_{T^\delta}(\eta\rest\alpha)|>1\}$ contains a club
and thus the set 
\[A\stackrel{\rm def}{=}\big\{\alpha<\lambda:\alpha\mbox{ is limit and }
(\forall\delta<\alpha)(|\suc_{T^\delta}(\eta\rest\alpha)|>1)\mbox{ and }
\eta\rest\alpha\in F_\alpha\big\}\] 
contains a club. But if $\alpha\in A$, then also $|\suc_{T^\alpha}(\eta\rest
\alpha)|>1$ and hence $|\suc_S(\eta\rest\alpha)|>1$ (remember (ii)+(iv)). 

\noindent {\sc Case}\quad $\ell=2$\\
Suppose that a sequence $\langle s_i:i<j\rangle\subseteq S$ is
$\vtl$--increasing and $|\suc_S(s_i)|>1$ for all $i<j$. Let
$s=\bigcup\limits_{i<j} s_i$ and $\delta=\lh(s)$. Then also
$|\suc_{T^\delta}(s_i)|>1$ (for all $i<j$) and hence
$|\suc_{T^\delta}(s)|>1$. By (v)+(iv)+(iii)+(i) we easily conclude 
$|\suc_S(s)|>1$ (note that $s\trianglelefteq t$ for some $t\in
F_\delta$).  

\noindent {\sc Case}\quad $\ell=3$\\
Let $C_\delta\subseteq\lambda$ be a club such that
\[\alpha\in C_\delta\ \&\ t\in T^\delta\cap {}^\alpha\lambda\quad
\Rightarrow\quad |\suc_{T^\delta}(t)|>1.\]
Set $C=\mathop{\triangle}\limits_{\delta<\lambda} C_\delta$. Then for each
limit $\alpha\in C$ and $t\in S\cap {}^\alpha\lambda$ we have that
$|\suc_{T^\delta}(t)|>1$ for all $\delta<\alpha$, and hence also
$|\suc_{T^\alpha}(t)|>1$ (by (ii)). Invoking (v)+(iv) we see that
$|\suc_S(t)|>1$ whenever $t\in S$, $\lh(t)\in C$ is limit. 

\noindent {\sc Case}\quad $\ell=4$\\
If $\mrot(S)\vtl s\in S$, then $|\suc_{T^\delta}(s)|>1$ for all
$\delta<\lambda$ and hence $|\suc_S(s)|>1$ (remember (v)). 
\end{proof}

\begin{proposition}
\label{prop2.4}
Let $\lambda$ be a strongly inaccessible cardinal, $Y_\delta=
{}^\delta\delta$ for $\delta<\lambda$ and $\bar{Y}=\langle Y_\delta: \delta
<\lambda\rangle$. Then the forcing notions $(\qell,\leq,\leq_\pr)$ for
$\ell\in\{2,3,4\}$ are $\lambda$-semi--purely proper over $\bar{Y}$. 
\end{proposition}

\begin{proof}
Let $1<\ell\leq 4$, $T\in\qell$. Consider the following strategy $\st$ of
Generic in the game $\Game^{\rm main}_{\bar{Y}}(T,\qell,\leq,\leq_\pr)$.  

In the course of the play, in addition to her innings $\langle
T_{\delta,\eta}:\eta\in Y_\delta\rangle$, Generic chooses also sets
$A_\delta\subseteq Y_\delta$ and conditions $T^\delta\in \qell$ so that
$T^\delta$ is decided before the stage $\delta$ of the game.  Suppose 
that the two players arrived to a stage $\delta<\lambda$.  If $\delta=0$
then Generic lets $T^0=T$ and if $\delta$ is limit, then she puts
$T^\delta=\bigcap\limits_{i<\delta} T^i$ (in both cases
$T^\delta\in\qell$). Now Generic determines $A_\delta$ and $\langle
T_{\delta,\eta}:\eta\in Y_\delta\rangle$ as follows. She sets
$A_\delta=T^\delta\cap Y_\delta$ and then she lets $\langle
T_{\delta,\eta}:\eta\in Y_\delta\rangle \subseteq\qell$ be a system of
pairwise incompatible conditions chosen so that 
\begin{itemize}
\item if $\eta\in A_\delta$ then $T_{\delta,\eta}=(T^\delta)_\eta$.
\end{itemize}
Generic's inning at this stage is $\langle T_{\delta,\eta}:\eta\in Y_\delta
\rangle$. After this Antigeneric answers with a system $\langle
S_{\delta,\eta}:\eta\in Y_\delta\rangle\subseteq\qell$ such that
$T_{\delta,\eta}\leq_\pr S_{\delta,\eta}$, and then Generic writes aside  
\[T^{\delta+1}\stackrel{\rm def}{=}\big\{t\in T^\delta: \big(\exists \eta
\in A_\delta\big)\big(\eta\trianglelefteq t\ \&\ t\in S_{\delta,\eta} \big) 
\mbox{ or }\big(\forall\alpha\leq \lh(t)\big)\big(t\rest\alpha\notin
A_\delta\big)\big\}.\] 
It should be clear that $T^{\delta+1}$ is a condition in $\qell$.  

After the play is finished and sequences 
\[\langle T_{\delta,\eta},S_{\delta,\eta}:\delta<\lambda\ \&\ \eta\in
Y_\delta\rangle\quad\mbox{ and }\quad \langle A_\delta,T^\delta:
\delta<\lambda\rangle\]
have been constructed, Generic lets
\[S=\bigcap\limits_{\delta<\lambda} T^\delta\subseteq T.\]

\begin{claim}
  \label{cl1}
$S\in\qell$ is aux-generic over $\bar{S}=\langle S_{\delta,\eta}:
\delta<\lambda\ \&\ \eta\in Y_\delta\rangle$. 
\end{claim}

\begin{proof}[Proof of the Claim]
First note that the sequence $\langle T^\delta, F_\delta=T^\delta\cap 
{}^\delta\lambda:\delta<\lambda\rangle$ satisfies the assumptions of Lemma  
\ref{pre2.5} and hence $S\in\qell$.   

Now we consider the three possible cases separately. 

\noindent {\sc Case}\quad $\ell=2$.\\
Let us describe a strategy $\st^*$ of COM in the game
$\auxzero(S,\bar{S},\qone,\leq,\leq_\pr)$. It instructs COM to play as
follows. Aside, COM picks also ordinals $\xi_\delta<\lambda$ so that after
arriving at a stage $\delta<\lambda$, when a sequence  $\langle 
(S_\alpha,A_\alpha,\eta_\alpha,S'_\alpha),\xi_\alpha:\alpha<\delta \rangle$
has been already constructed, she answers with $S_\delta,A_\delta,
\eta_\delta$ (and $\xi_\delta$) chosen so that the following demands are
satisfied. 
\begin{enumerate}
\item[(A)] $S_0=S$, $\xi_0=\lh(\mrot(S))+942$, $A_0=[\xi_0,\lambda)$ and
  $\eta_0=\langle\rangle$.
\item[(B)] If $\delta$ is a successor ordinal, say $\delta=\alpha+1$, then
\[\eta_\alpha\vtl\eta_\delta\in S_\alpha'\cap {}^\delta\lambda,\quad
\xi_\delta=\xi_\alpha+\sup(\eta_\delta(i):i<\delta)+ \lh\big(\mrot\big(
(S_\alpha')_{\eta_\delta}\big)\big)+942,\]
$A_\delta=[\xi_\delta,\lambda)$ and
$S_\delta=(S_\alpha')_{\eta_\delta}$. (Note: then we will also have
$\eta_\delta\in S_\delta'$.)   
\item[(C)] If $\delta$ is a limit ordinal, then $\eta_\delta=
\bigcup\limits_{\alpha<\delta}\eta_\alpha$, $\xi_\delta=\sup(\xi_\alpha:
\alpha<\delta)+942$,  $A_\delta=[\xi_\delta,\lambda)$ and $S_\delta=
\bigcap\limits_{\alpha<\delta}S_\alpha=\bigcap\limits_{\alpha<\delta}
S'_\alpha=\bigcap\limits_{\alpha<\delta}\big(S'_\alpha \big)_{\eta_\delta}$.
(Note: then we will also have $\eta_\delta\in S_\delta'$.)    
\end{enumerate}
Note that if $\langle (S_\alpha,A_\alpha,\eta_\alpha,S'_\alpha):\alpha<\lambda
\rangle$ is a play in which COM follows $\st^*$ and $\delta\in
\mathop{\triangle}\limits_{\alpha<\lambda} A_\alpha$ is a limit ordinal,
then $\eta_\delta\in Y_\delta\cap S$ and it is a limit of splitting points in
$S_\alpha'$, so also $|\suc_{S_\alpha'}(\eta_\delta)|>1$ for all
$\alpha<\delta$. Therefore, by (C), $\eta_\delta$ is a splitting node in
$S_\delta$ (and in $S$ as well). It follows from the description of the
$\delta$th move of Generic in $\Game^{\rm main}_{\bar{Y}}(T,\qone, \leq,
\leq_\pr)$, that  
\[(T^\delta)_{\eta_\delta}\leq_\pr
S_{\delta,\eta_\delta} =(T^{\delta+1})_{\eta_\delta}\leq_\pr
(S)_{\eta_\delta}\leq_\pr S_\delta.\]
Consequently, $\st^*$ is a winning strategy for COM. 

\noindent {\sc Case}\quad $\ell=3$.\\
The winning strategy $\st^*$ of COM in the game $\auxzero(S,\bar{S},\qthree,
\leq,\leq_\pr)$ is almost exactly the same as in the previous case. The only
difference is that now COM shrinks the answers $S'_\alpha$ of INC to members
of $\bbQ^{3,*}_\lambda$ pretending they were played in the game. The
argument that this is a winning strategy is exactly the same as before (as
$\bbQ^{3,*}_\lambda\subseteq \qone$). 

\noindent {\sc Case}\quad $\ell=4$. Similar.
\end{proof}
\end{proof}

The forcing notions considered above can be slightly generalized by allowing
the use of filters other than the club filter on $\lambda$. The forcing
notions $\bbQ^{\bar{E}}_E$ of \cite[Definition 1.11]{RoSh:888} and
$\bbP^{\bar{E}}_E$ of \cite[Definition 4.2]{RoSh:888} follow this
pattern. However, to apply the iteration theorems of \cite{RoSh:888} we
need to assume that the filter $E$ controlling splittings along branches is
concentrated on a stationary co-stationary set. Therefore the case of $E$
being the club filter seems to be of a different character. Putting general
filters on the splitting nodes {\em only\/} and controlling the splitting
levels by the club filter leads to Definition \ref{defFilter}. 

The forcing notion $\bbQ^{2,\bar{E}}$ was studied by Brown and Groszek
\cite{BrGr06} who described when this forcing adds a generic of minimal
degree. 

\begin{definition}
\label{defFilter}
Suppose that $\bar{E}=\langle E_t:t\in {}^{{<}\lambda}\lambda\rangle$ is a
system of $({<}\lambda)$--complete filters on $\lambda$. (These could be
principal filters.) We define forcing notions $\bbQ^{\ell,\bar{E}}$ for
$\ell=1,2,3,4$ as follows:
\begin{enumerate}
\item {\bf A condition} in $\bbQ^{2,\bar{E}}$ is a complete $\lambda$--tree
$T\subseteq {}^{{<}\lambda}\lambda$ such that  
\begin{enumerate}
\item[(a)] if $t\in T$, then $|\suc_T(t)|=1$ or $\suc_T(t)\in E_t$, and  
\item[(b)] $(\forall t\in T)(\exists s\in T)(t\vtl s\ \&\ |\suc_T(s)|>1)$, and    
\item[(c)$^2$] if $\langle t_i:i<j\rangle\subseteq T$ is $\vtl$--increasing, 
  $|\suc_T(t_i)|>1$ for all $i<j$ and $t=\bigcup\limits_{i<j}t_i$, then ($t\in
  T$ and) $|\suc_T(t)|>1$,
\end{enumerate}
{\bf the order} $\leq$ of $\bbQ^{2,\bar{E}}$ is the inverse
inclusion, i.e., $T_1\leq T_2$ if and only if $T_2\subseteq T_1$.
\item Forcings notions $\bbQ^{1,\bar{E}},\bbQ^{3,\bar{E}}, \bbQ^{4,\bar{E}}$ 
  are defined analogously, but the demand (c)$^2$ is replaced by the
  respective (c)$^\ell$:
\begin{enumerate}
\item[(c)$^1$] for every $\lambda$--branch $\eta\in\lim_\lambda(T)$ the set
  $\{\alpha\in \lambda: |\suc_T(\eta\rest\alpha)|>1\}$ contains a club of
  $\lambda$,\\  
\item[(c)$^3$] for some club $C\subseteq \lambda$ we have 
\[(\forall t\in T)(\lh(t)\in C\ \Rightarrow\ |\suc_T(t)|>1),\]
\item[(c)$^4$] $(\forall t\in T)(\mrot(T)\vtl t\ \Rightarrow\
  |\suc_T(t)|>1)$. 
\end{enumerate}
\item For $\ell=1,2,3,4$ we define a binary relation $\leq_\pr$ on
  $\bbQ^{\ell,\bar{E}}$ by\\ 
$T_1\leq_\pr T_2$ if and only if $T_1\leq T_2$ and $\mrot(T_1)=\mrot(T_2)$. 
\end{enumerate}
\end{definition}

\begin{remark}
  \label{oldfor}
Since in Definition \ref{defFilter} we allow the filters $E_t$ to be
principal, we may fit some classical forcings into our schema. If
$E_t=\{\lambda\}$ for each $t\in {}^{{<}\lambda}\lambda$, then
$\bbQ^{4,\bar{E}}$ is the $\lambda$--Cohen forcing $\bbC_\lambda$
(see Definition \ref{def3.2}(1)) and $\bbQ^{2,\bar{E}}$ is the
forcing $\bbD_\lambda$ from \cite[Definition 4.9(b)]{RoSh:655}. If for each
$t\in {}^{{<}\lambda}\lambda$ we let $E_t$ be the filter of all subsets
of $\lambda$ including $\{0,1\}$, then the forcing notion
$\bbQ^{2,\bar{E}}$ will be equivalent with Kanamori's
$\lambda$--Sacks forcing of \cite[Definition 1.1]{Ka80}.  
\end{remark}

\begin{proposition}
\label{stuffFilter}
Let $\bar{E}=\langle E_t:t\in {}^{{<}\lambda}\lambda\rangle$ be a system
of $({<}\lambda)$--complete filters on $\lambda$ and $\ell\in\{1,2,3,4\}$. 
\begin{enumerate}
\item $(\bbQ^{\ell,\bar{E}},\leq,\leq_\pr)$ is a forcing notion with
  $\lambda$--complete semi-purity. Moreover, the relations
  $(\bbQ^{\ell,\bar{E}},\leq)$ and $(\bbQ^{\ell,\bar{E}}, \leq_\pr)$ are
  $(<\lambda)$--complete.    
\item If $\lambda$ is strongly inaccessible, $Y_\delta= {}^\delta\delta$ for
  $\delta<\lambda$ and $\bar{Y}=\langle Y_\delta: \delta <\lambda\rangle$,
  then the forcing notions $(\bbQ^{\ell,\bar{E}},\leq,\leq_\pr)$ for
  $\ell\in\{2,3,4\}$ are $\lambda$-semi--purely proper over $\bar{Y}$.  
\end{enumerate}
\end{proposition}

\begin{proof}
Same as \ref{prop2.2}, \ref{prop2.4}. 
\end{proof}

Close relatives of the forcing notions $\bbQ^{\ell,\bar{E}}$ were considered
in \cite[Section B.8]{RoSh:777} and \cite[Definition 4.6]{RoSh:888}. The
modification now is that we consider trees branching into less than
$\lambda$ successor nodes (but there are many successors from the point of
view of suitably complete filters).   

\begin{definition}
\label{boundedPQ}
Assume that 
\begin{itemize}
\item $\lambda$ is strongly inaccessible, $f:\lambda\longrightarrow\lambda$
  is a increasing function such that each $f(\alpha)$ is a regular uncountable
  cardinal and $\prod\limits_{\xi<\alpha} f(\xi)^{|\alpha|}<f(\alpha)$ (for
  $\alpha<\lambda$), 
\item $\bar{F}=\langle F_t:t\in\bigcup\limits_{\alpha<\lambda}
  \prod\limits_{\xi<\alpha} f(\xi)\rangle$ where $F_t$ is a
  ${<}f(\alpha)$--complete filter on $f(\alpha)$ whenever $t\in
  \prod\limits_{\xi<\alpha}f(\xi)$, $\alpha<\lambda$. 
\end{itemize}
\begin{enumerate}
\item We define a forcing notion $\bbQ^1_{f,\bar{F}}$ as follows.\\ 
{\bf A condition} in $\bbQ^1_{f,\bar{F}}$ is a complete
  $\lambda$--tree $T\subseteq \bigcup\limits_{\alpha<\lambda}
  \prod\limits_{\xi<\alpha}f(\xi)$ such that  
\begin{enumerate}
\item[(a)] for every $t\in T$, either $|\suc_T(t)|=1$ or $\suc_T(t)\in  
  F_t$, and
\item[(b)] $(\forall t\in T)(\exists s\in T)(t\vtl s\ \&\ |\suc_T(s)|>1)$,
  and 
\item[(c)$^1$] for every $\eta\in\lim_\lambda(T)$ the set $\{\alpha<\lambda: 
\suc_T(\eta\rest\alpha)\in F_{\eta\rest\alpha}\}$ contains a club of
$\lambda$ . 
\end{enumerate}
\noindent {\bf The order} of $\bbQ^1_{f,\bar{F}}$ is the reverse inclusion. 
\item Forcing notions $\bbQ^\ell_{f,\bar{F}}$ for $\ell=2,3,4$ are defined
  similarly, but the demand (c)$^1$ is replaced by the respective
  (c)$^\ell$:  
\begin{enumerate}
\item[(c)$^2$] if $\langle t_i:i<j\rangle\subseteq T$ is $\vtl$--increasing, 
$|\suc_T(t_i)|>1$ for all $i<j$ and $t=\bigcup\limits_{i<j}t_i$, then ($t\in 
T$ and) $|\suc_T(t)|>1$,
\item[(c)$^3$] for some club $C$ of $\lambda$ we have 
\[\big(\forall t\in T\big)\big(\lh(t)\in C\ \Rightarrow\ \suc_T(T)\in
F_t \big).\]
\item[(c)$^4$] $(\forall t\in T)(\mrot(T)\vtl t\ \Rightarrow\
  |\suc_T(t)|>1)$. 
\end{enumerate}
\item For $\ell=1,2,3,4$ and $\alpha<\lambda$ we define a binary relation
$\leq^\alpha_\pr$ on $\bbQ^\ell_{f,\bar{F}}$ by\\ 
$T_1\leq_\pr^\alpha T_2$ if and only if either $T_1=T_2$ or $T_1\leq T_2$,
$\mrot(T_1)=\mrot(T_2)$ and $\lh(\mrot(T_2))\geq \alpha$.  
\end{enumerate}
\end{definition}

\begin{proposition}
\label{boundedFilter}
Assume $\lambda,f,\bar{F}$ are as in \ref{boundedPQ}. 
\begin{enumerate}
\item $(\bbQ^\ell_{f,\bar{F}},\leq,\bar{\leq}_\pr)$ is a forcing notion with 
  $f$--complete semi-purity.  
\item If $Y_\delta=\prod\limits_{\xi<\delta}f(\xi)$ for $\delta<\lambda$ and 
  $\bar{Y}=\langle Y_\delta: \delta <\lambda\rangle$, then $\bar{Y}$ is an
  indexing sequence and the forcing notions $(\bbQ^\ell_{f,\bar{F}},
  \leq,\bar{\leq}_\pr)$ for $\ell\in\{2,3,4\}$ are $f$-semi--purely proper
  over $\bar{Y}$.     
\end{enumerate}
\end{proposition}

\begin{proof}
Similar to \ref{prop2.2}, \ref{prop2.4}. 
\end{proof}

\begin{observation}
  Let $\eta\in {}^\lambda\lambda$ and $Y_\alpha=\{\eta\rest\alpha\}$ for
  $\alpha<\lambda$. Suppose that $(\bbQ,\leq)$ is a strategically
  $({\leq}\lambda)$--complete forcing notion and let $\leq^\alpha_\pr$ be
  $\leq$ (for $\alpha<\lambda$). Then $\bbQ$ is $\lambda$--semi-purely
  proper over $\langle Y_\xi:\xi<\lambda\rangle$ and the club filter with
  $\langle\leq^\alpha_\pr:\alpha<\lambda\rangle$ witnessing this.
\end{observation}

\begin{corollary}
  \label{cor2.6}
  Let $\lambda$ be a strongly inaccessible cardinal. Suppose that $\bar{E}$
  is as in \ref{defFilter} and $f,\bar{F}$ are as in \ref{boundedPQ}. Let
  $\bar{\bbQ}= \langle \bbP_\xi,\name{\bbQ}_\xi:\xi<\gamma\rangle$ be a
  $\lambda$--support iteration such that for every $\xi<\gamma$ the iterand
  $\name{\bbQ}_\xi$ is either strategically $({\leq}\lambda)$--complete, or
  it is one of $\qone,\qthree, \qfour, \bbQ^{2,\bar{E}}, \bbQ^{3,\bar{E}},
  \bbQ^{4,\bar{E}}, \bbQ^2_{f,\bar{F}}, \bbQ^3_{f,\bar{F}},
  \bbQ^4_{f,\bar{F}}$. Then $\bbP_\gamma=\lim(\bar{\bbQ})$ is
  $\lambda$--proper in the standard sense.
\end{corollary}

\section{Are $\qone$, $\qtwo$ very different?}
The forcing notions $\qtwo$ and $\qone$ appear to be very close. In this
section we will show that, consistently, they are equivalent, but also
consistently, they may be different. 

\begin{lemma}
 \label{lem3.1}
Assume $T\in \bT^{\rm club}$ and consider $(T,\vtl)$ as a forcing
notion. Let $\name{\eta}$ be a $T$--name for the generic $\lambda$--branch
added by $T$. Suppose that  
\[\forces_T\mbox{`` the set }\{\alpha<\lambda:|\suc_T(\name{\eta}\rest
\alpha)|>1\}\mbox{ contains a club ''.}\]
Then there is $T^*\subseteq T$ such that $T^*\in \bbQ^2_\lambda$. 
\end{lemma}

\begin{proof}
Let $\name{C}$ be a $T$--name for a club of $\lambda$ such that 
\[\forces_T\mbox{`` }\big(\forall \alpha\in \name{C}\big) \big(|\suc_T(
\name{\eta}\rest\alpha)|>1\big)\mbox{ '',}\]
and put
\[S=\big\{t\in T:\lh(t)\mbox{ is a limit ordinal and }t\forces_T\mbox{`` }
 \big(\forall\alpha<\lh(t)\big)\big(\name{C}\cap (\alpha,\lh(t))\neq
\emptyset\big)\}\mbox{ ''.}\]
One easily verifies that
\begin{enumerate}
\item[(i)]  if $t\in S$, then $t\forces_T\mbox{`` }\lh(t)\in\name{C}\mbox{
    ''}$ and hence $|\suc_T(t)|>1$,  
\item[(ii)] if a sequence $\langle t_\alpha:\alpha<\alpha^*\rangle \subseteq
  S$ is $\vtl$--increasing, $\alpha^*<\lambda$, then 
  $\bigcup\limits_{\alpha<\alpha^*} t_\alpha\in S$,
\item[(iii)] $(\forall t\in T)(\exists s\in S)(t\vtl s)$.
\end{enumerate}
Consequently we may choose $T^*\subseteq T$ so that $T^*\in\bT^{\rm club}$
and for some fronts $F_\alpha$ of $T^*$ (for $\alpha<\lambda$) we have  
\begin{itemize}
\item $F_\alpha\subseteq S$, and if $\alpha<\beta<\lambda$, $t\in F_\beta$,
  then for some $s\in F_\alpha$ we have $s\vtl t$, 
\item if $\delta<\lambda$ is limit and $t_\alpha\in F_\alpha$ for
  $\alpha<\delta$ are such that $\alpha<\beta<\delta\ \Rightarrow t_\alpha
  \vtl t_\beta$, then $\bigcup\limits_{\alpha<\delta} t_\alpha\in F_\delta$,  
\item $|\suc_{T^*}(t)|>1$ if and only if $t\in
  \bigcup\limits_{\alpha<\lambda} F_\alpha$. 
\end{itemize}
Then also $T^*\in \bbQ^{1,*}_\lambda=\qone$ (remember Observations
\ref{obs2.4}, \ref{prop2.3}). 
\end{proof}

\begin{definition}
  \label{def3.2}
\begin{enumerate}
\item The $\lambda$--Cohen forcing notion $\bbC_\lambda$ is defined as
  follows:\\
{\bf a condition in $\bbC_\lambda$} is a sequence $\nu\in
{}^{<\lambda}\lambda$,\\
{\bf the order $\leq$ of $\bbC_\lambda$} is the extension of sequences
(i.e., $\nu_1\leq\nu_2$ if and only if $\nu_1\trianglelefteq \nu_2$). 
\item The axiom ${\rm Ax}^+_{\bbC_\lambda}$ is the following statement:\\
if $\name{S}$ is a $\bbC_\lambda$--name and $\forces_{\bbC_\lambda}\mbox{``
}\name{S}\mbox{ is a stationary subset of }\lambda\mbox{ ''}$, and
$\cO_\alpha\subseteq \bbC_\lambda$ are open dense sets (for
$\alpha<\lambda$) then there is a $\trianglelefteq$--directed
$\trianglelefteq$--downward closed set $H\subseteq \bbC_\lambda$ such that 
\begin{itemize}
\item $H\cap\cO_\alpha\neq \emptyset$ for all $\alpha<\lambda$, and
\item the interpretation $\name{S}[H]$ of the name $\name{S}$ is a
  stationary subset of $\lambda$. 
\end{itemize}
\end{enumerate}
\end{definition}

\begin{lemma}
  \label{lem3.3}
Let $T\in \bT^{\rm club}$. Then the following conditions are equivalent:
\begin{enumerate}
\item[(a)] there is $T^*\subseteq T$ such that $T^*\in \bbQ^2_\lambda$,  
\item[(b)] there is $T^*\subseteq T$ such that $T^*\in\bT^{\rm club}$ and 
\[\forces_{\bbC_\lambda}(\forall\eta\in {\lim}_\lambda(T^*))
(\mbox{the set }\{\delta{<}\lambda:|\suc_{T^*}(\eta\rest\delta)|>1\}
\mbox{ contains a club of }\lambda).\] 
\end{enumerate}
\end{lemma}

\begin{proof}
Assume (a). By the $({<}\lambda)$--completeness of $\bbC_\lambda$ we see
that $\forces_{\bbC_\lambda} T^*\in \bbQ^2_\lambda$, and hence
$\forces_{\bbC_\lambda}T^*\in \bbQ^1_\lambda$ (remember Observation
\ref{prop2.3}). Consequently (b) follows. 
\medskip

Now assume (b). Since $(T^*,\trianglelefteq)$ (as a forcing notion) is
isomorphic with $\bbC_\lambda$ we have
\[\forces_{T^*}(\forall\eta\in {\lim}_\lambda(T^*))(\mbox{the set } 
\{\delta{<}\lambda:|\suc_T(\eta\rest\delta)|>1\}\mbox{ contains a club of 
}\lambda),\]
so in particular
\[\forces_{T^*}\mbox{`` the set }\{\delta<\lambda:|\suc_T(\name{\eta}
\rest\delta)|>1\}\mbox{ contains a club of }\lambda\mbox{ '',}\] 
where $\name{\eta}$ is a $T^*$--name for the generic $\lambda$--branch. It 
follows now from Lemma \ref{lem3.1} that (a) holds.
\end{proof}

\begin{proposition}
\label{prop3.4}
Assume ${\rm Ax}^+_{\bbC_\lambda}$. Then $\bbQ^2_\lambda$ is a dense
subset of $\qtwo$ (so the forcing notions $\qtwo,\qone$ are equivalent).   
\end{proposition}

\begin{proof}
Let $T\in\qtwo$ and let us consider $(T,\trianglelefteq)$ as a forcing
notion. Let $\name{S}$ be a $T$--name given by
\[\forces_T\name{S}=\{\delta<\lambda:|\suc_T(\name{\eta}\rest\delta)|>1\}\] 
where $\name{\eta}$ is a $T$--name for the generic $\lambda$--branch. Ask
the following question
\begin{itemize}
\item Does $\forces_T\mbox{`` }\name{S}\mbox{ contains a club of
  }\lambda\mbox{ ''}$ ?
\end{itemize}
If the answer is ``yes'', then by Lemma \ref{lem3.1} there is $T^*\subseteq
T$ such that $T^*\in \bbQ^2_\lambda$. 

So assume that the answer to our question is ``not''. Then there is $t\in T$
such that 
\[t\forces_T\mbox{`` }\lambda\setminus\name{S}\mbox{ is stationary }\mbox{
  ''.}\] 
Let $\name{S}'=\{(\check{\alpha},s): s\in T\mbox{ and }\alpha= \lh(s)\mbox{
  and }|\suc_T(s)|=1\}$. Then $\name{S}'$ is a $T$--name for a subset of
$\lambda$ and $\forces_T\name{S}'=\lambda\setminus \name{S}$. Therefore,
$t\forces_T$`` $\name{S}'$ is stationary '' and since the forcing notion $T$
above $t$ is isomorphic with $\bbC_\lambda$, we may use the assumption of
${\rm Ax}^+_{\bbC_\lambda}$ to pick a $\trianglelefteq$--directed
$\trianglelefteq$--downward closed set $H\subseteq T$ such that $t\in H$ and   
\begin{itemize}
\item $H\cap\{s\in T:\lh(s)>\alpha\}\neq \emptyset$ for all
  $\alpha<\lambda$, and 
\item $\name{S}'[H]$ is stationary in $\lambda$.
\end{itemize}
Then for each $\alpha<\lambda$ the intersection $H\cap {}^\alpha\lambda$ is
a singleton, say $H\cap {}^\alpha\lambda=\{\eta_\alpha\}$, and 
\begin{itemize}
\item if $\alpha<\beta$ then $\eta_\alpha\vtl \eta_\beta$, and
\item $\alpha\in\name{S}'[H]$ if and only if $|\suc_T(\eta_\alpha)|=1$.
\end{itemize}
Let $\eta=\bigcup\limits_{\alpha<\lambda} \eta_\alpha$. Then
$\eta\in\lim_\lambda(T)$ and the set $\{\alpha<\lambda: |\suc_T(\eta \rest
\alpha)|=1\}$ is stationary, contradicting $T\in\qtwo$.  
\end{proof}

\begin{proposition}
\label{notdense}
It is consistent that $\qone$ is not dense in $\qtwo$. 
\end{proposition}

\begin{proof}
We will build a $({<}\lambda)$--strategically complete $\lambda^+$--cc
forcing notion forcing that $\qone$ is not dense in $\qtwo$. It will be
obtained by means of a $({<}\lambda)$--support iteration of length
$2^\lambda$. First, we define a forcing notion $\bbQ_0$:\\
{\bf a condition in} $\bbQ_0$ is a tree $q\subseteq {}^{{<}\lambda}\lambda$
such that $|q|<\lambda$;\\
{\bf the order $\leq=\leq_{\bbQ_0}$ of $\bbQ_0$} is defined by\\
$q\leq q'$ if and only if $q\subseteq q'$ and $(\forall\eta\in q)(|\suc_q
(\eta)|=1\ \Rightarrow\ |\suc_{q'}(\eta)|=1)$.  

Plainly, $\bbQ_0$ is a $({<}\lambda)$--complete forcing notion of size
$\lambda$. Let $\name{T}_0$ be a $\bbQ_0$--name such that
$\forces_{\bbQ_0}\mbox{`` }\name{T}_0=\bigcup\name{G}_{\bbQ_0}$ ''. Then 
\[\forces_{\bbQ_0}\mbox{`` }\name{T}_0\in \bT^{\rm club}\mbox{ and }
\big(\forall\eta\in\name{T}_0\big)\big(|\suc_{\name{T}_0}(\eta)|>1\
\Rightarrow\ \suc_{\name{T}_0}(\eta)=\lambda\big)\mbox{ ''.}\]

For a set $A\subseteq\lambda$ let $\bbQ^A$ be the forcing notion shooting a
club through $A$. Thus\\ 
{\bf a condition in} $\bbQ^A$ is a closed bounded set $c$ included in $A$,\\
{\bf the order $\leq=\leq_{\bbQ^A}$ of $\bbQ^A$} is defined by\\
$c\leq c'$ if and only if $c=c'\cap\big(\max(c)+1\big)$.

Now we inductively define a $({<}\lambda)$--support iteration
$\bar{\bbQ}=\langle\bbP_\xi,\dbQ_\xi:\xi<2^\lambda\rangle$ and a sequence
$\langle \name{A}_\xi,\name{\eta}_\xi:\xi<2^\lambda\rangle$ so that the
following demands are satisfied. 
\begin{enumerate}
\item[(i)] $\bbQ_0$ is the forcing notion defined above, $\name{T}_0$ is the
  $\bbQ_0$--name for the generic tree added by $\bbQ_0$.
\item[(ii)] $\bbP_\xi$ is strategically $({<}\lambda)$--complete, satisfies
  $\lambda^+$--cc and has a dense subset of size $2^\lambda$ (for each
  $\xi\leq 2^\lambda$). 
\item[(iii)] $\name{\eta}_\xi,\name{A}_\xi$ are $\bbP_\xi$--names such that  
\[\forces_{\bbP_\xi}\mbox{`` }\name{\eta}_\xi\in {\lim}_\lambda(\name{T}_0)
\mbox{ and }\name{A}_\xi=\{\alpha<\lambda:|\suc_{\name{T}_0}(\name{\eta}_\xi
\rest\alpha)|>1\}\mbox{ ''.}\]
\item[(iv)] $\forces_{\bbP_\xi}$`` $\dbQ_\xi=\bbQ^{\name{A}_\xi}$ ''. 
\item[(v)] If $\name{\eta}$ is a $\bbP_{2^\lambda}$--name for a member of
  $\lim_\lambda(\name{T}_0)$, then for some $\xi<2^\lambda$ we have
  $\forces_{\bbP_\xi}$`` $\name{\eta}=\name{\eta}_\xi$ ''. 
\end{enumerate}

Clause (ii) will be shown soon, but with it in hand using a bookkeeping
device we can take care of clause (v). Then the iteration $\bar{\bbQ}$ will
be fully determined. So let us argue for clause (ii) (assuming that the
iteration is constructed so that clauses (i), (iii) and (iv) are satisfied).

For $0<\xi\leq 2^\lambda$ we let $\bbP^*_\xi$ consist of all conditions
$p\in\bbP_\xi$ such that $0\in\dom(p)$ and for some limit ordinal
$\delta^p<\lambda$, for each $i\in\dom(p)\setminus\{0\}$ we have: 
\begin{enumerate}
\item[(a)] $p(0)\subseteq {}^{{\leq}\delta^p+1}\lambda$ and for some
  $\eta_{p,i}\in {}^{\delta^p}\lambda\cap p(0)$, $p\rest
  i\forces_{\bbP_i}$`` $\name{\eta}_i\rest \delta^p=\eta_{p,i}$ '', 
\item[(b)] $p(i)$ is a closed subset of $\delta^p+1$ (not just a 
  $\bbP_i$--name) and $\delta^p\in p(i)$,
\item[(c)] if $\beta\in p(i)$, then $|\suc_{p(0)}(\eta_{p,i}\rest
  \beta)|>1$.   
\end{enumerate}

\begin{claim}
\label{cl3}
\begin{enumerate}
\item If $p,p'\in\bbP^*_\xi$, then $p\leq_{\bbP_\xi} p'$ if and only if
$\dom(p)\subseteq\dom(p')$, $p(0)\leq_{\bbQ_0} p'(0)$ and $(\forall
  i\in\dom(p)\setminus\{0\})(p(i)=p'(i)\cap (\delta^p+1))$. 
\item $|\bbP^*_\xi|=\lambda\cdot |\xi|^{<\lambda}$.
\item $\bbP^*_\xi$ is a $({<}\lambda)$--complete $\lambda^+$--cc subforcing
  of $\bbP_\xi$. Moreover, if $\langle p_\alpha:\alpha<\gamma\rangle$ is a 
  $\leq_{\bbP_\xi}$--increasing sequence of members of $\bbP^*_\xi$,
  $\gamma<\lambda$, then there is $p\in\bbP^*_\xi$ such that
  $p_\alpha\leq_{\bbP_\xi} p$ for all $\alpha<\lambda$ and
  $\delta^p=\sup(\delta^{p_\alpha}:\alpha<\gamma)$.   
\item $\bbP_\xi^*$ is dense in $\bbP_\xi$. Moreover, for every $p\in
  \bbP_\xi$ and $\alpha<\lambda$ there is $q\in\bbP^*_\xi$ such that
  $p\leq_{\bbP_\xi} q$ and $\delta^q>\alpha$. 
\end{enumerate}
\end{claim}

\begin{proof}[Proof of the Claim]
(1), (2), (3)\quad  Straightforward.

\noindent (4)\quad Induction on $\xi\in (0,2^\lambda]$.

\noindent{\sc Case}\quad $\xi=\xi_0+1$\\
Let $p\in\bbP_\xi$. Construct inductively a sequence $\langle p_n:
n<\omega\rangle\subseteq \bbP^*_{\xi_0}$ such that for each $n<\omega$ we
have   
\begin{itemize}
\item $p\rest\xi_0\leq_{\bbP_{\xi_0}} p_n\leq_{\bbP_{\xi_0}} p_{n+1}$ and
  $\alpha<\delta^{p_n}<\delta^{p_{n+1}}$, 
\item for some closed set $c\subseteq\delta^{p_0}$ we have $p_0
  \forces_{\bbP_{\xi_0}}$`` $p(\xi_0)=c$ '', 
\item for some sequence $\nu_n\in {}^{\delta^{p_n}}\lambda\cap p_{n+1}(0)$
  we have $p_{n+1}\forces_{\bbP_{\xi_0}}$``$\name{\eta}_{\xi_0}\rest
  \delta^{p_n} =\nu_n$''. 
\end{itemize}
(The construction is clearly possible by our inductive hypothesis.) Now we 
define a condition $q\in \bbP_\xi^*$. We declare that $\dom(q)=
\bigcup\limits_{n<\omega}\dom(p_n)\cup\{\xi_0\}$ and for $i\in\dom(q)
\setminus \{0,\xi_0\}$ we set $\eta_{q,i}=\bigcup\{\eta_{p_n,i}: i\in
\dom(p_n),\ n\in\omega\}$ and we also put $\eta_{q,\xi_0}=
\bigcup\limits_{n<\omega}\nu_n$. We define  
\begin{itemize}
\item $\delta^q=\sup\limits_{n<\omega}\delta^{p_n}$, 
$q(0)=\bigcup\limits_{n<\omega}
p_n(0)\cup\{\eta_{q,i},\eta_{q,i}\conc\langle 0\rangle,\eta_{q,i}\conc
\langle 1\rangle:i\in \dom(q)\setminus\{0\}\}$,
\item $q(i)=\bigcup\limits_{n<\omega} p_n(i)\cup \{\delta^q\}$ for $i\in
  \dom(q)\setminus\{0,\xi_0\}$ and 
\item $q(\xi_0)=c\cup\{\delta^q\}$. 
\end{itemize}
One easily verifies that $q\in\bbP^*_\xi$ and it is stronger than $p$. 

\noindent{\sc Case}\quad $\xi$ is limit and $\cf(\xi)<\lambda$\\
Let $p\in\bbP_\xi$. Fix an increasing sequence $\langle\xi_\vare:
\vare<\cf(\xi)\rangle\subseteq \xi$ cofinal in $\xi$ and then use the
inductive assumption (and properties of an iteration) to choose inductively
a sequence $\langle p_\vare:\vare<\cf(\xi)\rangle$ such that for each
$\vare<\vare'< \cf(\xi)$ we have 
\[p_\vare\in\bbP_{\xi_\vare}^*,\quad \alpha<\delta^{p_\vare}<
\delta^{p_{\vare'}}\quad\mbox{ and }\quad p\rest\xi_\vare
\leq_{\bbP_{\xi_\vare}} p_\vare\leq_{\bbP_{\xi_\vare}} p_{\vare'}\rest
\xi_\vare.\]  
Then define a condition $q\in\bbP_\xi^*$ as follows. Declare that
$\dom(q)=\bigcup\limits_{\vare<\cf(\xi)}\dom(p_\vare)$ and for $i\in \dom(q)
\setminus \{0\}$ set $\eta_{q,i}= \bigcup\{\eta_{p_\vare,i}:i\in
\dom(p_\vare),\ \vare<\cf(\xi)\}$. Put 
\begin{itemize}
\item $\delta^q=\sup\limits_{\vare<\cf(\xi)}\delta^{p_\vare}$, 
\item $q(0)=\bigcup\limits_{\vare<\cf(\xi)} p_\vare(0)\cup\{\eta_{q,i},
  \eta_{q,i}\conc\langle 0\rangle,\eta_{q,i}\conc\langle 1\rangle:i\in
  \dom(q)\setminus\{0\}\}$, 
\item $q(i)=\bigcup\limits_{\vare<\cf(\xi)} p_\vare(i)\cup \{\delta^q\}$ for
  $i\in\dom(q)\setminus\{0\}$.
\end{itemize}

\noindent{\sc Case}\quad $\xi$ is limit and $\cf(\xi)\geq\lambda$\\
Immediate as then $\bbP_\xi=\bigcup\limits_{\zeta<\xi}\bbP_\zeta$. 
\end{proof}

It follows from \ref{cl3} that the clause (ii) of the construction of the
iteration is satisfied. In particular, the limit $\bbP_{2^\lambda}$ is
strategically $({<}\lambda)$--complete $\lambda^+$--cc and has a dense
subset of size $2^\lambda$. It should be also clear that
$\forces_{\bbP_{2^\lambda}}$`` $\name{T}_0\in\qtwo$ '' (remember
(iii)--(v)).  

\begin{claim}
\label{cl4}
$\forces_{\bbP_{2^\lambda}}$ `` $\name{T}_0$ contains no tree from 
$\qone$ ''.   
\end{claim}

\begin{proof}[Proof of the Claim]
Suppose towards contradiction that $\name{T}$ is a $\bbP_{2^\lambda}$--name
such that 
\[p\forces_{\bbP_{2^\lambda}}\mbox{`` }\name{T}\in \qone\ \mbox{ and }\
\name{T}\subseteq \name{T}_0\mbox{ ''}\qquad\mbox{ for some
}p\in\bbP_{2^\lambda}.\]  
Note that, by \ref{cl3}(3,4),  
\begin{enumerate}
\item[$(*)$] if $p\leq_{\bbP_{2^\lambda}} q$, $\nu\in {}^{{<}\lambda}
\lambda$, $q\forces_{\bbP_{2^{\lambda}}}\nu\in \name{T}$, and
$\kappa<\lambda$ and $\name{\rho}_i$ are $\bbP_{2^\lambda}$--names for
members of ${}^\lambda\lambda$ (for $i<\kappa$),\\
then there are $q^*\in\bbP^*_{2^\lambda}$ and $\nu^*\in q^*(0)$ such that
$q\leq_{\bbP_{2^\lambda}} q^*$, $\nu\vtl\nu^*$ and  
\[q^*\forces_{\bbP_{2^\lambda}}\mbox{`` }\nu^*\in\name{T}\ \ \&\ \
|\suc_{\name{T}}(\nu^*)|>1\ \ \&\ \ (\forall i<\kappa)(\name{\rho}_i \rest 
\lh(\nu^*)\neq\nu^*)\mbox{ ''.}\]  
\end{enumerate}
Using $(*)$ repeatedly $\omega$ times we may construct a sequence $\langle
p_n,\nu^*_n:n<\omega\rangle$ such that 
\begin{itemize}
\item $p_n\in \bbP^*_{2^\lambda}$, $p\leq_{\bbP_{2^\lambda}} 
  p_n\leq_{\bbP_{2^\lambda}} p_{n+1}$, $\delta^{p_n}<\delta^{p_{n+1}}$, and  
\item $\nu^*_n\in {}^{{<}\lambda}\lambda$, $\nu^*_n\in p_{n+1}(0)$, $\nu^*_n  
  \vtl\nu^*_{n+1}$, and    
\[p_{n+1}\forces_{\bbP_{2^\lambda}}\mbox{`` }\nu_n^*\in\name{T}\ \&\
  |\suc_{\name{T}}(\nu^*_n)|>1\ \&\ (\forall i\!\in\!\dom(p_n)\setminus  
  \{0\})(\name{\eta}_i\rest\lh(\nu^*_n)\neq\nu^*_n)\mbox{ ''.}\] 
\end{itemize}
Then we define a condition $q\in \bbP_\xi^*$: we declare that
$\dom(q)=\bigcup\limits_{n<\omega}\dom(p_n)$ and for $i\in\dom(q)\setminus
\{0\}$ we set $\eta_{q,i}= \bigcup\{\eta_{p_n,i}:i\in\dom(p_n),\ n\in
\omega\}$. We also put $\nu^*=\bigcup\limits_{n<\omega} \nu^*_n$,
$\delta^q=\sup\limits_{n<\omega}\delta^{p_n}$,  and then:
\begin{itemize}
\item $q(0)=\bigcup\limits_{n<\omega}p_n(0)\cup\{\eta_{q,i},\eta_{q,i} \conc 
  \langle 0\rangle,\eta_{q,i}\conc\langle 1\rangle:i\in \dom(q)\setminus
  \{0\}\}\cup \{\nu^*, \nu^*\conc\langle 0\rangle\}$,
\item $q(i)=\bigcup\limits_{n<\omega} p_n(i)\cup \{\delta^q\}$ for $i\in
  \dom(q)\setminus\{0\}$.
\end{itemize}
Note that $\nu^*\notin\{\eta_{q,i}\rest\lh(\nu^*):i\in\dom(q)
\setminus\{0\}\}$ and $\lh(\nu^*)\leq\delta^q=\lh(\eta_{q,i})$ for
$i\in\dom(q)\setminus\{0\}$. Now we easily check that
$q\in\bbP^*_{2^\lambda}$ is stronger than $p$ and it forces that $\nu^*\in
\name{T}$ is a limit of splitting points of $\name{T}$, but itself it is not
a splitting point (even in $\name{T}_0$). A contradiction with
$p\forces\name{T}\in\bbQ^2_\lambda$.   
\end{proof}
\end{proof}

\begin{proposition}
\label{propNOTiso}
Assume that the complete Boolean algebras ${\rm RO}(\qtwo)$ and ${\rm
  RO}(\qone)$ are isomorphic. Then $\qone$ is a dense subset of
$\qtwo$. 
\end{proposition}

\begin{proof}
Since ${\rm RO}(\qtwo)$ and ${\rm RO}(\qone)$ are isomorphic, we may find
$\bbQ^{3-\ell}_\lambda$--names $\name{H}_\ell,\name{\eta}_\ell$ (for
$\ell=1,2$) such that 
\begin{enumerate}
\item[$(\boxdot)_1$] $\forces_{\qell}$`` $\name{H}_{3-\ell}\subseteq
  \bbQ^{3-\ell}_\lambda$ is generic over $\bV$ and $\name{\eta}_{3-\ell}
  \in{}^\lambda\lambda$ is the corresponding generic branch '',
\item[$(\boxdot)_2$] if $G_\ell\subseteq\qell$ is generic over $\bV$ and
  $G_{3-\ell}=\name{H}_{3-\ell}[G_\ell]$, then
  $G_\ell=\name{H}_\ell[G_{3-\ell}]$.  
\end{enumerate}
Consider $\qtwo\times\qone$ with the product order and for $\ell=1,2$ put 
\[R_\ell=\{(T_1,T_2)\in\qtwo\times\qone: T_\ell\forces_{\qell} T_{3-\ell}\in
\name{H}_{3-\ell}\}\]
and $R=R_1\cap R_2$. 

\begin{claim}
\label{cl2}
$R$ is a dense subset of both $R_1$ and $R_2$.  
\end{claim}

\begin{proof}[Proof of the Claim]
First note that 
\begin{enumerate}
\item[$(\boxdot)_3$] if $(T_1,T_2)\in\qtwo\times\qone$ and
  $T_1\nVdash_{\qtwo}\mbox{`` }T_2\notin \name{H}_2\mbox{ ''}$, then there
  is $T^*_1\geq T_1$ such that $(T^*_1,T_2)\in R_1$ (and symmetrically when
  the roles of 1 and 2 are interchanged).   
\end{enumerate}
Also,
\begin{enumerate}
\item[$(\boxdot)_4$] if $(T_1,T_2)\in R_\ell$, $\ell=1,2$, then $T_{3-\ell}
  \nVdash_{\bbQ^{3-\ell}_\lambda} T_\ell\notin \name{H}_\ell$. 
\end{enumerate}
[Why? Assume towards contradiction that
$T_{3-\ell}\forces_{\bbQ^{3-\ell}_\lambda} T_\ell\notin\name{H}_\ell$. Let
$G_\ell\subseteq\qell$ be a generic over $\bV$ such that $T_\ell\in
G_\ell$. Put $G_{3-\ell}=\name{H}_{3-\ell}[G_\ell]$. Then
$G_{3-\ell}\subseteq\bbQ^{3-\ell}_\lambda$ is generic over $\bV$ and
$\name{H}_\ell[G_{3-\ell}]=G_\ell$. Since $(T_1,T_2)\in R_\ell$ we know
$T_{3-\ell}\in\name{H}_{3-\ell}[G_\ell]$ and hence (by our assumption
towards contradiction) $T_\ell\notin\name{H}_\ell[G_{3-\ell}]=G_\ell$,
contradicting the choice of $G_\ell$.] 

Now suppose $(T_1,T_2)\in R_\ell$, $\ell\in\{1,2\}$. Choose inductively a
sequence $\langle (T^n_1,T^n_2):n<\omega\rangle$ such that $(T^0_1,T^0_2)
=(T_1,T_2)$ for all $n<\omega$:
\begin{itemize}
\item if $n$ is even, then $(T^n_1,T^n_2)\in R_\ell$,
\item if $n$ is odd, then $(T^n_1,T^n_2)\in R_{3-\ell}$,
\item $(T^n_1,T^n_2)\leq (T^{n+1}_1,T^{n+1}_2)$. 
\end{itemize}
By $(\boxdot)_3+(\boxdot)_4$ there are no problems with carrying out the
inductive process. Put 
\[T^\omega_1=\bigcap_{n<\omega} T^n_1\quad\mbox{ and }\quad T^\omega_2=
\bigcap_{n<\omega} T^n_2.\]
Then $T^\omega_\ell$ is the least upper bound of $\langle
T^n_\ell:n<\omega\rangle$ and hence easily $(T^\omega_1,T^\omega_2)\in R$,
$(T_1,T_2)\leq (T^\omega_1,T^\omega_2)$. 
\end{proof}

\begin{claim}
\label{cl5}
Let $\ell\in\{1,2\}$, $T\in\qell$. Then there is $T^*\geq T$ such that for
some $\nu\in {}^{{<}\lambda}\lambda$ we have 
\[\lh(\nu)=\lh(\mrot(T^*))\quad \mbox{ and }\quad T^*\forces_{\qell}
\mbox{`` }\nu\vtl\name{\eta}_{3-\ell}  \mbox{ ''.}\] 
\end{claim}

\begin{proof}[Proof of the Claim]
By induction on $\alpha<\lambda$ choose a sequence $\langle T_\alpha:\alpha
<\lambda\rangle$ so that for all $\alpha<\beta<\lambda$ we have  
\begin{enumerate}
\item[$(\boxdot)_5$] $T_\alpha\leq T_\beta$, $\mrot(T_\alpha)\vtl
  \mrot(T_\beta)$ and 
\item[$(\boxdot)_6$] $T_{\alpha+1}$ forces a value to
  $\name{\eta}_{3-\ell}\rest \lh(\mrot(T_\alpha))$, and  
\item[$(\boxdot)_7$] if $\alpha$ is limit then $T_\alpha=\bigcap\limits_{
    \xi<\alpha} T_\xi$. 
\end{enumerate}
Let $\eta=\bigcup\limits_{\alpha<\lambda}\mrot(T_\alpha)\in
{}^\lambda\lambda$. Then $\eta\in\lim_\lambda(T_\alpha)$ for each
$\alpha<\lambda$ so the sets $\{\delta<\lambda:|\suc_{T_\alpha}
(\eta\rest\delta)|>1\}$ contain clubs (for each
$\alpha<\lambda$). Consequently we may pick limit $\delta<\lambda$ such that
$|\suc_{T_\alpha}(\eta\rest\delta)|>1$ for all $\alpha<\delta$. Then also,
by $(\boxdot)_7$, $\eta\rest\delta=\mrot(T_\delta)$ and clearly (by
$(\boxdot)_6$) $T_\delta$ forces a value to $\name{\eta}_{3-\ell}\rest
\delta$.   
\end{proof}

\begin{claim}
\label{cl6}
Let $\ell\in\{1,2\}$, $T\in\qell$. Then there is $T^*\geq T$ such that for
every $t\in T^*$, for some $\nu\in {}^{{<}\lambda}\lambda$ we have 
\[\lh(\nu)=\lh(t)\quad \mbox{ and }\quad (T^*)_t\forces_{\qell}
\mbox{`` }\nu\vtl\name{\eta}_{3-\ell} \mbox{ ''.}\] 
\end{claim}

\begin{proof}[Proof of the Claim]
We choose inductively conditions $T_\alpha\in\qell$ and fronts
$F_\alpha$ of $T_\alpha$ so that for all $\alpha<\beta<\lambda$:
\begin{enumerate}
\item[$(\boxdot)_8$] $T_0=T$, $F_0=\{\langle\rangle\}$, 
\item[$(\boxdot)_9$] $T_\alpha\leq T_\beta$, $F_\alpha\subseteq T_\beta$
  and $(\forall t\in F_\beta)(\exists i<\lh(t))(t\rest i\in F_\alpha)$, 
\item[$(\boxdot)_{10}$] if $\alpha$ is limit, then $T_\alpha=
  \bigcap\limits_{\xi<\alpha} T_\xi$ and $F_\alpha=\big\{t\in T_\alpha: 
  (\forall \xi<\alpha)(\exists i<\lh(t))(t\rest i\in F_\xi)$ and
  $(\forall i<\lh(t))(\exists \xi<\alpha)(\exists j<\lh(t))(i<j\ \&\
  t\rest j\in F_\xi)\big\}$,
\item[$(\boxdot)_{11}$] if $t\in F_\alpha$, $\zeta\in\suc_{T_\alpha}(t)$
  then $t\conc\langle\zeta\rangle\in T_{\alpha+1}$ and for some
  $s=s_{t,\zeta}\in F_{\alpha+1}$ we have 
  \[\mrot\big((T_{\alpha+1})_{t\conc\langle\zeta\rangle}\big)=s\ \mbox{
  and }\ (T_{\alpha+1})_s\mbox{ forces a value to }
\name{\eta}_{3-\ell} \rest \lh(s),\]   
\item[$(\boxdot)_{12}$] $F_{\alpha+1}=\{s_{t,\zeta}:t\in F_\alpha\ \&\
  \zeta\in\suc_{T_\alpha}(t)\}$ (where $s_{t,\zeta}$ are determined by
  $(\boxdot)_{11}$). 
\end{enumerate}
It should be clear that the construction is possible (at successor stages
use \ref{cl5}). Set $T^*=\bigcap\limits_{\alpha<\lambda} T_\alpha$. By Lemma
\ref{pre2.5} we know that $T^*\in\qell$ and $F_\alpha\subseteq T^*$ are
fronts of $T^*$ (for $\alpha<\lambda$). Moreover, 
\begin{enumerate}
\item[$(\boxdot)_{13}$] if $s\in T^*$ and $|\suc_{T^*}(s)|>1$, then $s\in
  \bigcup\limits_{\alpha<\lambda} F_\alpha$. 
\end{enumerate}
It follows from $(\boxdot)_{11}+(\boxdot)_{12}+ (\boxdot)_{10}$ that for
every $t\in F_\alpha$ the condition $(T^*)_t$ forces a value to
$\name{\eta}_{3-\ell}\rest \lh(t)$. If $t\in T^*\setminus\bigcup 
\limits_{\alpha<\lambda} F_\alpha$, then choose the shortest $s\in
\bigcup\limits_{\alpha<\lambda} F_\alpha$ such that $t\vtl s$. Then
$(T^*)_t= (T^*)_s$ (remember $(\boxdot)_{13}$) and hence in particular the
condition $(T^*)_t$ forces a value to $\name{\eta}_{3-\ell}\rest
\lh(t)$. 
\end{proof}

Now suppose that $T_1\in\qtwo$. Use Claim \ref{cl6} to choose a condition
$T^*_1\in\qtwo$ such that $T_1\leq T^*_1$ and 
\begin{enumerate}
\item[$(\boxdot)^{T^*_1,2}_{14}$] for every $t\in T^*_1$ the condition
  $(T^*_1)_t$ forces a value to $\name{\eta}_2\rest \lh(t)$. 
\end{enumerate}
Then use Claim \ref{cl2} to pick $(T'_1,T'_2)\in R$ so that $T^*_1\leq
T'_1$. Note that then also $(\boxdot)^{T'_1,2}_{14}$ holds. Apply Claim
\ref{cl6} to $T'_2$ and $\ell=2$ to find a condition $T''_2\geq T'_2$ such 
that the suitable demand $(\boxdot)^{T''_2,1}_{14}$ holds, and then use
Claim \ref{cl2} again to choose $(T^+_1,T^+_2)\in R$ such that $T^+_1\geq
T'_1$ and $T^+_2\geq T''_2$. Note that then 
\begin{enumerate}
\item[$(\boxdot)^{T^+_\ell,3-\ell}_{14}$] for every $t\in T^+_\ell$ the
  condition $(T^+_\ell)_t$ forces a value to $\name{\eta}_{3-\ell}\rest
  \lh(t)$.   
\end{enumerate}
For $\ell=1,2$ and $t\in T^+_\ell$ let $\varphi_\ell(t)\in
{}^{{<}\lambda}\lambda$ be such that $\lh(\varphi_\ell(t))=\lh(t)$ and
$(T^+_\ell)_t\forces$``$\varphi_\ell(t)\vtl\name{\eta}_{3-\ell}$''. Since
$(T^+_1,T^+_2)\in R$ we know that 
\begin{enumerate}
\item[$(\boxdot)^\ell_{15}$] $\varphi_\ell(t)\in T^+_{3-\ell}$ for each
  $t\in T^+_\ell$  
\end{enumerate}
and by $(\boxdot)_2$ we also have 
\begin{enumerate}
\item[$(\boxdot)_{16}$] $\varphi_1\comp \varphi_2$ is the identity on
  $T^+_2$ and $\varphi_2\comp \varphi_1$ is the identity on
  $T^+_1$. Moreover,\\ 
if $t\in T^+_1$, then $((T^+_1)_t,(T^+_2)_{\varphi_1(t)})\in R$ and\\ 
if $s\in T^+_2$, then $((T^+_1)_{\varphi_2(s)},(T^+_2)_s)\in R$. 
\end{enumerate}
Thus $\varphi_\ell:T^+_\ell\longrightarrow T^+_{3-\ell}$ is a bijection
preserving levels and the extension relation $\vtl$, and $\varphi_1$ is the 
inverse of $\varphi_2$. Consequently, $t\in T^+_\ell$ is a splitting of
$T^+_\ell$ if and only if $\varphi_\ell(t)$ is a splitting in
$T^+_{3-\ell}$. Therefore we may conclude that $T^+_1\in\qone$.
\end{proof}

\begin{conclusion}
It is consistent that the forcing notions $\qtwo,\qone$ are not equivalent.   
\end{conclusion}

\begin{proof}
By Propositions \ref{notdense} and \ref{propNOTiso}.
\end{proof}


\begin{thebibliography}{10}

\bibitem{BrGr06}
Elizabeth~T. Brown and Marcia~J. Groszek.
\newblock Uncountable superperfect forcing and minimality.
\newblock {\em Ann. Pure Appl. Logic}, 144:73--82, 2006.

\bibitem{Ei03}
Todd Eisworth.
\newblock {On iterated forcing for successors of regular cardinals}.
\newblock {\em Fundamenta Mathematicae}, 179:249--266, 2003.

\bibitem{FrHoZdxx}
Sy-David Friedman, Radek Honzik, and Lyubomyr Zdomskyy.
\newblock {Fusion and large cardinal preservation}.
\newblock {\em Annals of Pure and Applied Logic}, accepted, 2012.

\bibitem{FrZd10}
Sy-David Friedman and Lyubomyr Zdomskyy.
\newblock {Measurable cardinals and the cofinality of the symmetric group}.
\newblock {\em Fundamenta Mathematicae}, 207:101--122, 2010.

\bibitem{J}
Thomas Jech.
\newblock {\em {Set theory}}.
\newblock Springer Monographs in Mathematics. Springer-Verlag, Berlin, 2003.
\newblock The third millennium edition, revised and expanded.

\bibitem{Ka80}
Akihiro Kanamori.
\newblock {Perfect-set forcing for uncountable cardinals}.
\newblock {\em Annals of Mathematical Logic}, 19:97--114, 1980.

\bibitem{RoSh:1001}
Andrzej Roslanowski and Saharon Shelah.
\newblock {The last forcing standing with diamonds}.
\newblock {\em preprint}.

\bibitem{RoSh:655}
Andrzej Roslanowski and Saharon Shelah.
\newblock {Iteration of $\lambda$-complete forcing notions not collapsing
  $\lambda^+$.}
\newblock {\em {International Journal of Mathematics and Mathematical
  Sciences}}, 28:63--82, 2001.
\newblock math.LO/9906024.

\bibitem{RoSh:860}
Andrzej Roslanowski and Saharon Shelah.
\newblock {Reasonably complete forcing notions}.
\newblock {\em Quaderni di Matematica}, 17:195--239, 2006.
\newblock math.LO/0508272.

\bibitem{RoSh:777}
Andrzej Roslanowski and Saharon Shelah.
\newblock {Sheva-Sheva-Sheva: Large Creatures}.
\newblock {\em Israel Journal of Mathematics}, 159:109--174, 2007.
\newblock math.LO/0210205.

\bibitem{RoSh:888}
Andrzej Roslanowski and Saharon Shelah.
\newblock {Lords of the iteration}.
\newblock In {\em Set Theory and Its Applications}, volume 533 of {\em
  Contemporary Mathematics (CONM)}, pages 287--330. Amer. Math. Soc., 2011.
\newblock math.LO/0611131.

\bibitem{RoSh:890}
Andrzej Roslanowski and Saharon Shelah.
\newblock {Reasonable ultrafilters, again}.
\newblock {\em Notre Dame Journal of Formal Logic}, 52:113--147, 2011.
\newblock math.LO/0605067.

\bibitem{Sh:587}
Saharon Shelah.
\newblock {Not collapsing cardinals $\leq\kappa$ in $(<\kappa)$--support
  iterations}.
\newblock {\em {Israel Journal of Mathematics}}, 136:29--115, 2003.
\newblock math.LO/9707225.

\bibitem{Sh:667}
Saharon Shelah.
\newblock {Successor of singulars: combinatorics and not collapsing cardinals
  $\leq\kappa$ in $(<\kappa)$-support iterations}.
\newblock {\em {Israel Journal of Mathematics}}, {134}:127--155, 2003.
\newblock math.LO/9808140.

\end{thebibliography}

\end{document}